\documentclass{gtpart}

\usepackage{pinlabel}
\usepackage[all]{xy}
\usepackage{graphicx}
\usepackage{float}
\usepackage{subfigure}
\usepackage{enumerate}

\title{Non-hyperbolic solutions to tangle equations \\involving composite links}

\author{\normalfont Jingling Yang}
\givenname{Jingling}
\surname{Yang}
\email{}
\address{Department of Mathematics\\
The Chinese University of Hong Kong\\\newline
Hong Kong}
\urladdr{}

\arxivreference{}  
\arxivpassword{}   

\newtheorem{thm1}{Theorem}
\newtheorem{thm}{Theorem}[section]
\newtheorem{lem}[thm]{Lemma}

\newtheorem{conj}[thm]{Conjecture}
\newtheorem{cor}[thm]{Corollary}
\newtheorem{prop}[thm]{Proposition}

\theoremstyle{definition}
\newtheorem{defn}[thm]{Definition}
\newtheorem*{rem}{Remark}

\makeop{Homo}
\numberwithin{equation}{section}

\def\ack{\section*{Acknowledgements}%
  \addtocontents{toc}{\protect\vspace{6pt}}%
  \addcontentsline{toc}{section}{Acknowledgements}%
}

\begin{document}

\begin{abstract}    
Solving tangle equations is deeply connected with studying enzyme action on DNA. The main goal of this paper is to solve the system of tangle equations $N(O+X_1)=b_1$ and $N(O+X_2)=b_2 \# b_3$, where $X_1$ and $X_2$ are rational tangles, and $b_i$ is a 2-bridge link, for $i=1,2,3$, with $b_2$ and $b_3$ nontrivial. We solve this system of equations under the assumption $\widetilde{O}$, the double branched cover of $O$, is not hyperbolic, i.e.$O$ is not $\pi$-hyperbolic. Besides, we also deal with tangle equations involving 2-bridge links only under the assumption $O$ is an algebraic tangle.
\end{abstract}

\maketitle


\section{Introduction and main theorem}

In order to study enzyme action on DNA, Ernst and Sumners \cite{Sumners} introduced tangle model, which models DNA as a knot/link and regards the enzyme action on the DNA knot/link as a tangle replacement. Fortunately, most of the DNA knot/link belong to the mathematically well-known class of 2-bridge knots/links (4-plat knots/links). Besides, based on biological experiments, it is possible to obtain a composite knot/link as a product of some enzyme action in certain situations. Therefore, the central purpose of this paper is to solve the following system of tangle equations, and the main theorem is given as follows:

\begin{thm1}
\label{mainthm}
Suppose
\begin{align*}
N(O+X_1)&=b_1          \\
N(O+X_2)&=b_2 \# b_3,
\end{align*}
where $X_1$ and $X_2$ are rational tangles, and $b_i$ is a 2-bridge link, for $i=1,2,3$, with $b_2$ and $b_3$ nontrivial. Suppose $\widetilde{O}$ is irreducible toroidal but not Seifert fibered. then the system of tangle equations has solutions if and only if one of the following holds:\\
(\romannumeral1) There exist 3 pairs of relatively prime integers $(a,b)$, $(p_1,q_1)$, and $(p,q)$ satisfying $0<\frac{b}{a}\leq 1 (or \,\, \frac{b}{a}=\frac{1}{0}=\infty), p>1,p_1>1,|aq_1-bp_1|>1$, and $q-pp_1q_1=1$ such that $b_1=b(a,b)$ and $b_2 \# b_3 =b(p,q) \#  b(a+ap_1q_1p-bp_1^2p, b+aq_1^2p-bp_1q_1p)$. Solutions up to equivalence are shown as the following:
\begin{figure}[H]
\centering
\subfigure[]{
\includegraphics[width=0.21\textwidth]{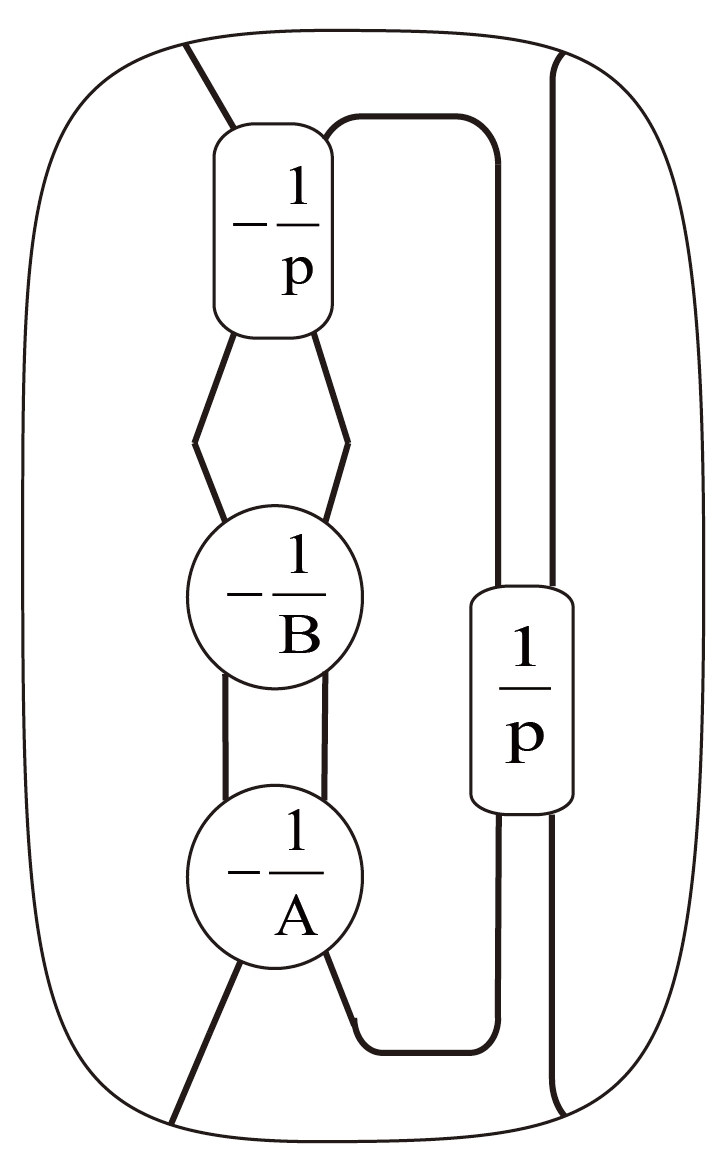}}\quad
\subfigure[]{
\includegraphics[width=0.21\textwidth]{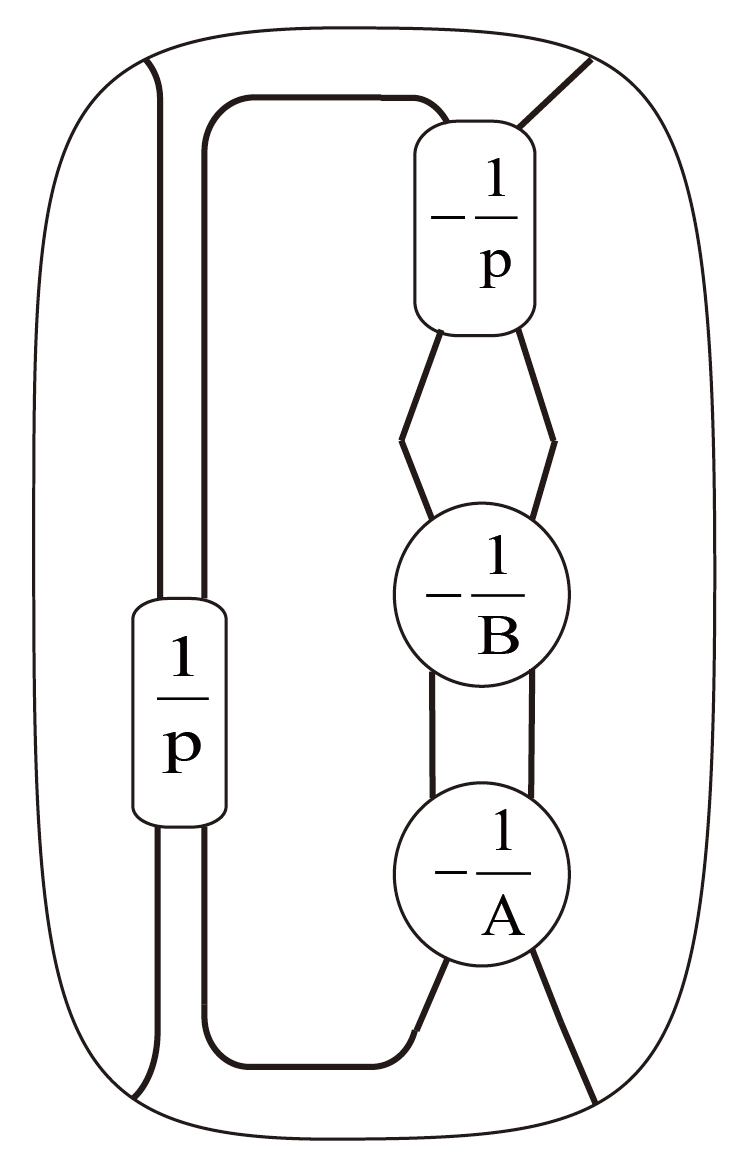}}\quad
\subfigure{
\includegraphics[width=0.15\textwidth]{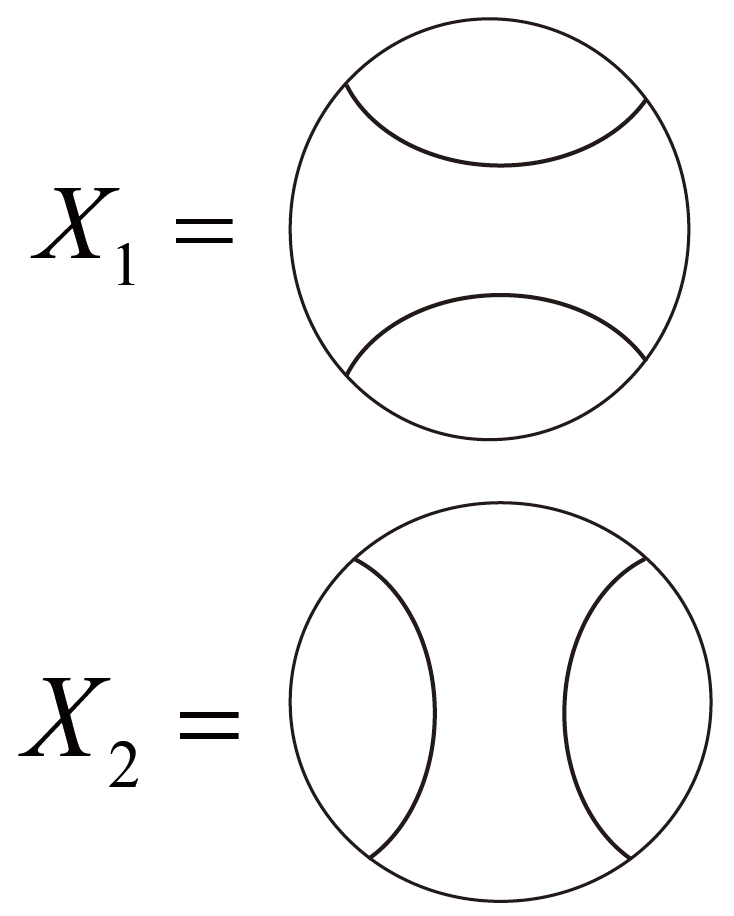}}
\caption{$O=$ the tangle in (a) or (b), where $A=\frac{-e}{p_1}$, $B=\frac{ad-be}{aq_1-bp_1}$ (or $A=\frac{ad-be}{aq_1-bp_1}$, $B=\frac{-e}{p_1}$), and $p_1d-q_1e=1$ with $d,e\in \mathbb{Z}$. $X_1=0$-tangle and $X_2=\infty$-tangle. (Note that choosing different $d$ and $e$ such that $p_1d-q_1e=1$ has no effect on the tangle $O$.)}
\end{figure}
(\romannumeral2) There exist 3 pairs of relatively prime integers $(a,b)$, $(p_1,q_1)$, and $(p,q)$ satisfying $0<\frac{b}{a}\leq 1 (or \,\, \frac{b}{a}=\frac{1}{0}=\infty), p>1,p_1>1,|aq_1-bp_1|>1$, and $q-pp_1q_1=-1$ such that $b_1=b(a,b)$ and $b_2 \# b_3 =b(p,q) \#  b(a-ap_1q_1p+bp_1^2p, b-aq_1^2p+bp_1q_1p)$. Solutions up to equivalence are shown as the following:
\begin{figure}[H]
\centering
\subfigure[]{
\includegraphics[width=0.21\textwidth]{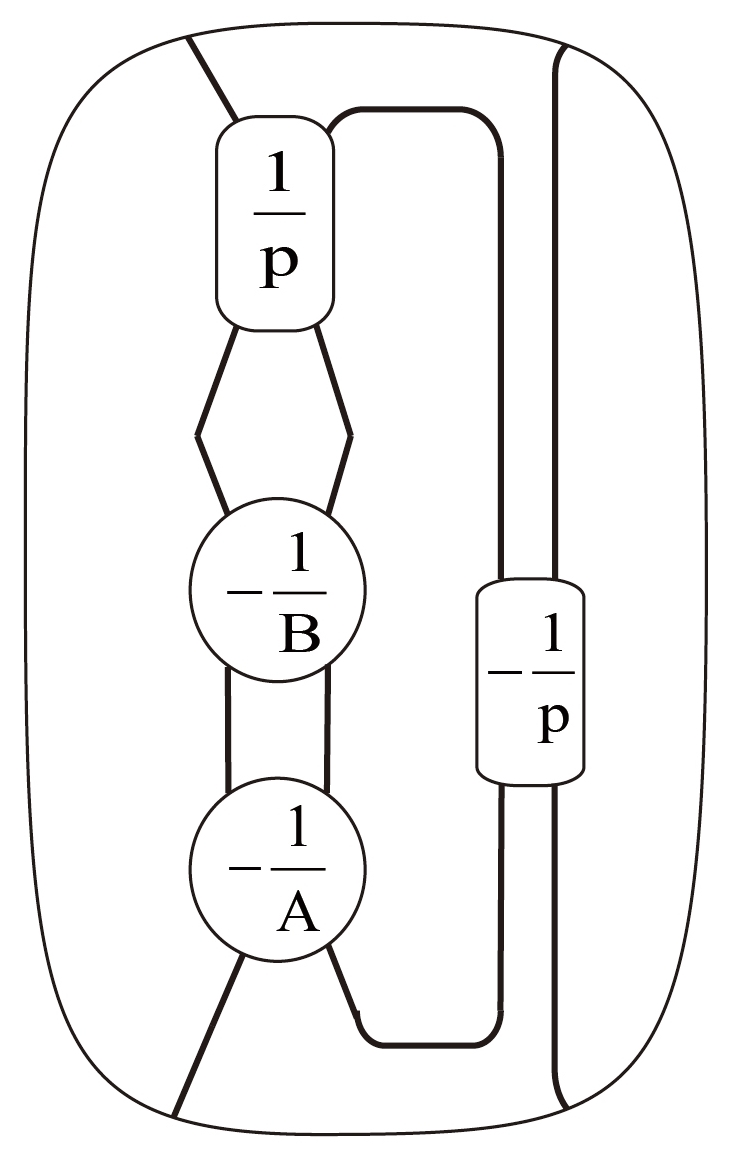}}\quad
\subfigure[]{
\includegraphics[width=0.21\textwidth]{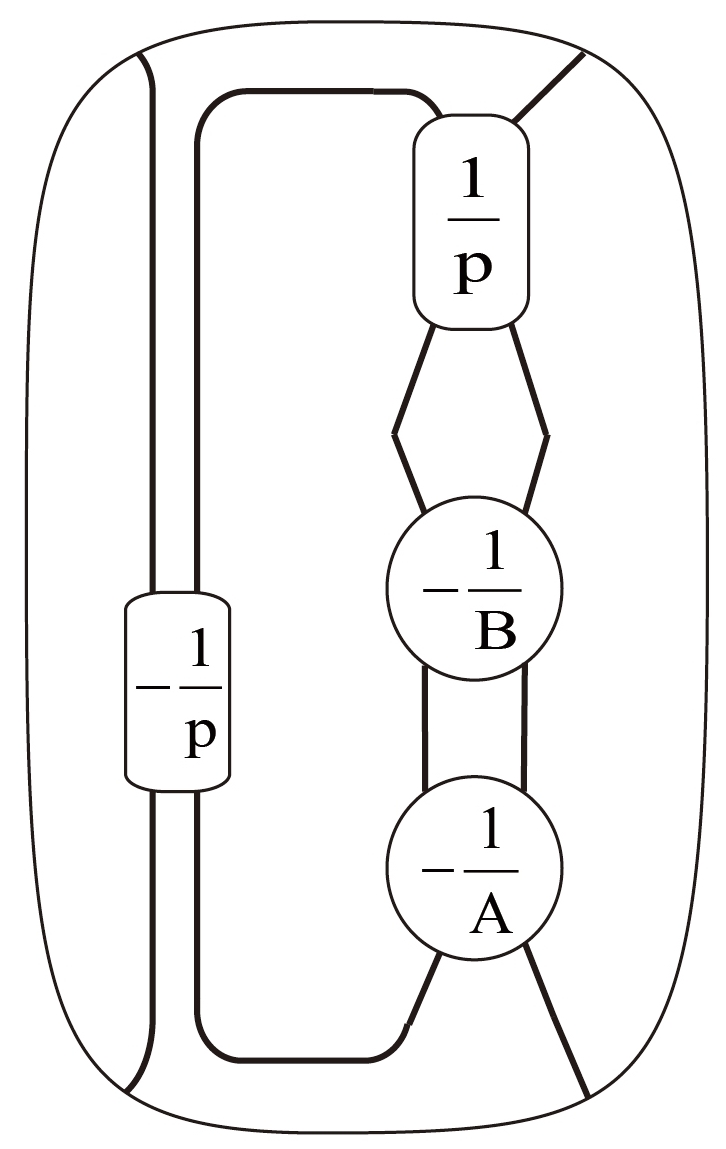}}\quad
\subfigure{
\includegraphics[width=0.15\textwidth]{solution2.jpg}}
\caption{$O=$ the tangle in (a) or (b), where $A=\frac{-e}{p_1}$, $B=\frac{ad-be}{aq_1-bp_1}$ (or $A=\frac{ad-be}{aq_1-bp_1}$, $B=\frac{-e}{p_1}$), and $p_1d-q_1e=1$ with $d,e\in \mathbb{Z}$. $X_1=0$-tangle and $X_2=\infty$-tangle. (Note that choosing different $d$ and $e$ such that $p_1d-q_1e=1$ has no effect on the tangle $O$.)}
\end{figure}

\end{thm1}

As a matter of fact, the key point to solve tangle equations is lifting to the double branched covers. When we take the "sum" of two tangles $O$ and $X$, it results in a link $K$. Lifting to their double branched covers, this operation on tangles $O$ and $X$ induces a gluing of the boundaries of their respective double branched covers, which produces a 3-manifold $\widetilde{K}$, the double cover of $S^3$ branched over $K$. The double branched cover of a rational tangle is a solid torus, so a rational tangle replacement induces a Dehn surgery on a 3-manifold. Here the 3-manifold is a lens space or a connected sum of two lens spaces since the double cover of $S^3$ branched over a 2-bridge link is a lens space. Thus, our problem turns out to be finding knots in a lens space admitting a surgery to give a connected sum of two lens spaces. Kenneth L. Baker \cite{Baker} proposed a lens space version of cabling conjecture stating that if a knot in a lens space admits a surgery to a non-prime 3-manifold then the knot is either lying in a ball, or a knot with Seifert fibered exterior, or a cabled knot, or hyperbolic. He proved the conjecture when the knot is non-hyperbolic in \cite{Baker}. Buck and Mauricio's paper \cite{Buck} addresses the problem in two cases, (1)the double branched cover of the tangle $O$, denoted by $\widetilde{O}$, is a Seifert fiber space and (2)$\widetilde{O}$ is reducible, namely $\widetilde{O}$ is the complement of a knot with Seifert fibered exterior and $\widetilde{O}$ is the complement of a knot in a ball. This paper mainly solves the aforementioned system of equations when $\widetilde{O}$ is the complement of a cabled knot, i.e.\,$\widetilde{O}$ is irreducible toroidal but not Seifert fibered. We also give solutions for the Seifert fibered case, which comprise a special case excluded in \cite{Buck}. In fact, we have a different definition of tangle from \cite{Buck}, since we allow tangle to have circles embedded in, so more solutions are given in this paper.

The following outlines the proof of Theorem \ref{mainthm}. First of all, we show that the cabled knot is a cable of a torus knot lying in one of the solid tori in the lens space, and the complement of this cabled knot is a graph manifold with only one essential torus. Then by studying Dehn surgeries along this cabled knot, we obtain all possible products $b_2 \# b_3$ when $b_1$ is given. Split the graph manifold along the essential torus into two pieces, both of which are Seifert fiber spaces. We can easily find two tangles whose double branched covers are the two pieces respectively, since it is not hard to find a Montesinos tangle (or a Montesinos pair) with this type of Seifert fiber spaces as the double branched cover. By gluing the two tangles together, we get a tangle whose double branched cover is this graph manifold. Finally, by studying the involutions on this graph manifold, we prove that any tangle whose double branched cover is this graph manifold is homeomorphic to the tangle we construct. Connecting with the analysis about Dehn fillings on this graph manifold, we give all the solutions $(O,X_1,X_2)$ of the tangle equations above when $\widetilde{O}$ is the complement of a cabled knot.

This paper is structured as follows: Section \ref{Preliminaries} provides some preliminaries about tangles, 2-bridge link and their double branched covers. Section \ref{Solving tangle equations} gives the complete process of solving tangle equations above. In Section \ref{Some other tangle equations}, we solve the system of tangle equations involving 2-bridge links only under the assumption $O$ is an algebraic tangle by using the same method as solving the aforementioned tangle equations.

\ack{I would like to express my gratitude to my supervisor Professor Zhongtao Wu for his professional guidance and helpful discussion. He introduced this subject to me and gave me many valuable suggestions. I would also like to thank Erica Flapan for introducing me algebraic tangle and Bonahon-Siebenmann theory. This work was partially supported by grants from the Research Grants Council of the Hong Kong Special Administrative Region, China (Project No. CUHK 14309016).}

\section{Preliminaries}
\label{Preliminaries}

\subsection{Tangles and 2-bridge link}

\begin{defn}
A {\it tangle} is a pair $(M,t)$ where $M$ is the complement of disjoint 3-balls in $S^3$ and $t$ is a properly embedded 1-manifold which intersects each boundary component of $M$ in 4 points.
\end{defn}

\begin{defn}
A {\it marked tangle} is a tangle $(M,t)$ with $k$ boundary components $S_1, \dots, S_k$ parameterized by a family of orientation-preserving homeomorphisms $\Phi=\cup_i \Phi_i: (\partial S_i, \partial S_i \cap t) \rightarrow (S^2, P)$, where $S^2$ is the unit sphere in $\mathbb{R}^3$ and $P=\{NE=(e^{\frac{\pi}{4}i},0),SE=(e^{\frac{-\pi}{4}i},0),SW=(e^{\frac{-3\pi}{4}i},0),NW=(e^{\frac{3\pi}{4}i},0)\}$. We denote a marked tangle as a triple $(M,t,\Phi)$.
\end{defn}

\begin{defn}
Two tangles $X=(M,t,\Phi)$ and $Y=(M',t',\Phi')$ are {\it isomorphic} if there exists a homeomorphism $H:(M,t)\rightarrow (M',t')$ such that $\Phi=\Phi'H$. $X$ and $Y$ are {\it homeomorphic} if there exists a homeomorphism $H:(M,t)\rightarrow (M',t')$.
\end{defn}

In this paper, we mainly deal with tangles in a 3-ball with a few exceptions. The following are some operations and useful results about tangles in a 3-ball.

\begin{defn}
Given two tangles $A$ and $B$ in a 3-ball, {\it tangle addition} is defined as shown in Figure \ref{tangleaddition}, denoted by $A+B$. {\it tangle multiplication} is defined as shown in Figure \ref{tanglemultiplication}, denoted by $A \times B$.
\end{defn}

\begin{figure}[H]
\centering
\subfigure[Tangle addition]{\label{tangleaddition}
\includegraphics[width=0.36\textwidth]{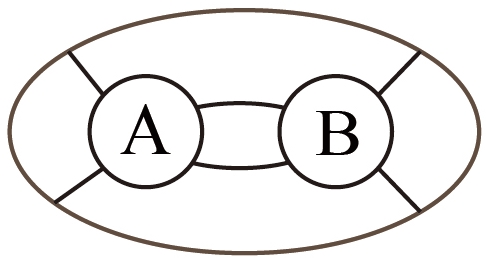}}\qquad\qquad
\subfigure[Tangle multiplication]{\label{tanglemultiplication}
\includegraphics[width=0.36\textwidth]{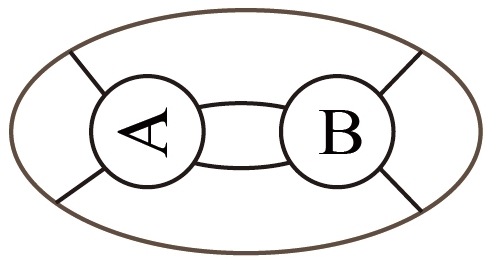}}
\caption{Tangle addition and multiplication}
\end{figure}

\begin{defn}
Given a tangle $A$ in a 3-ball, the {\it numerator closure} and {\it denominator closure} are defined as shown in Figure \ref{numerator closure} and Figure \ref{denominator closure} respectively, denoted by $N(A)$ and $D(A)$.
\end{defn}

\begin{figure}[H]
\centering
\subfigure[Numerator closure $N(A)$]{\label{numerator closure}
\includegraphics[width=0.29\textwidth]{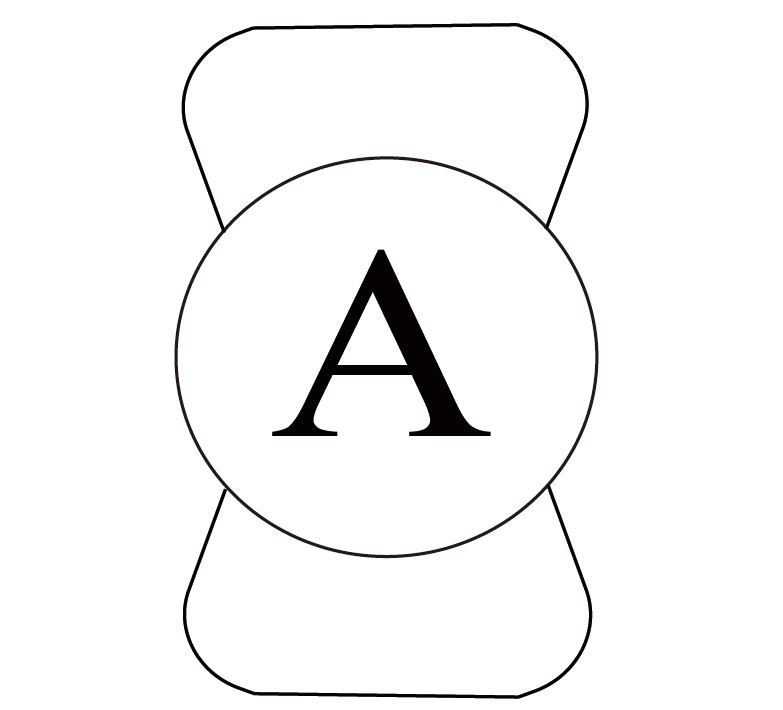}}\qquad\qquad
\subfigure[Denominator closure $D(A)$]{\label{denominator closure}
\includegraphics[width=0.36\textwidth]{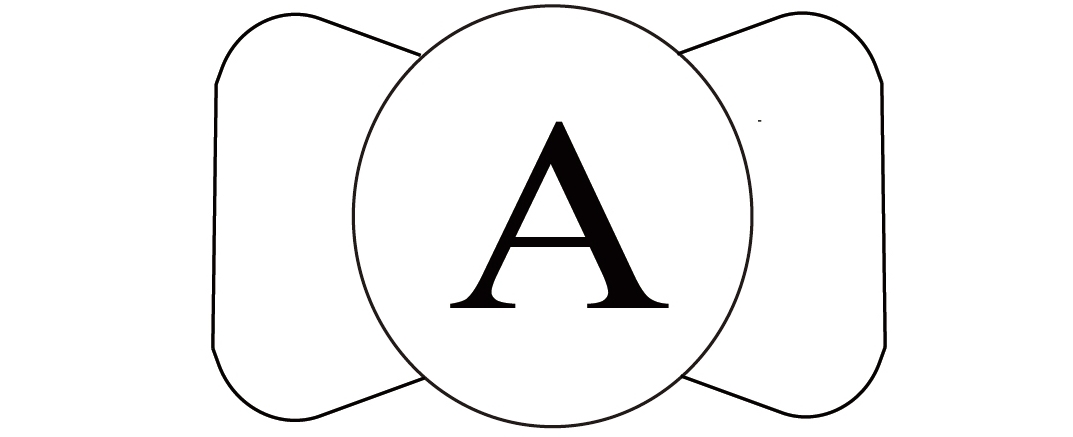}}
\caption{Numerator closure and denominator closure}
\end{figure}

\begin{rem}
$D(A+B)=D(A) \# D(B)$.
\end{rem}

\begin{defn}
The {\it circle product} of a tangle $A$ in a 3-ball and an integer vector $C=(c_1, c_2, \dots, c_n)$ is shown in Figure \ref{circleproduct}, denoted by $A \circ C$ or $A \circ (c_1, c_2, \dots, c_n)$.
\end{defn}

\begin{figure}[H]
\centering
\subfigure[$A \circ (c_1, c_2, \dots, c_n)$, when $n$ is even.]{\label{circleproduct1}
\includegraphics[width=0.335\textwidth]{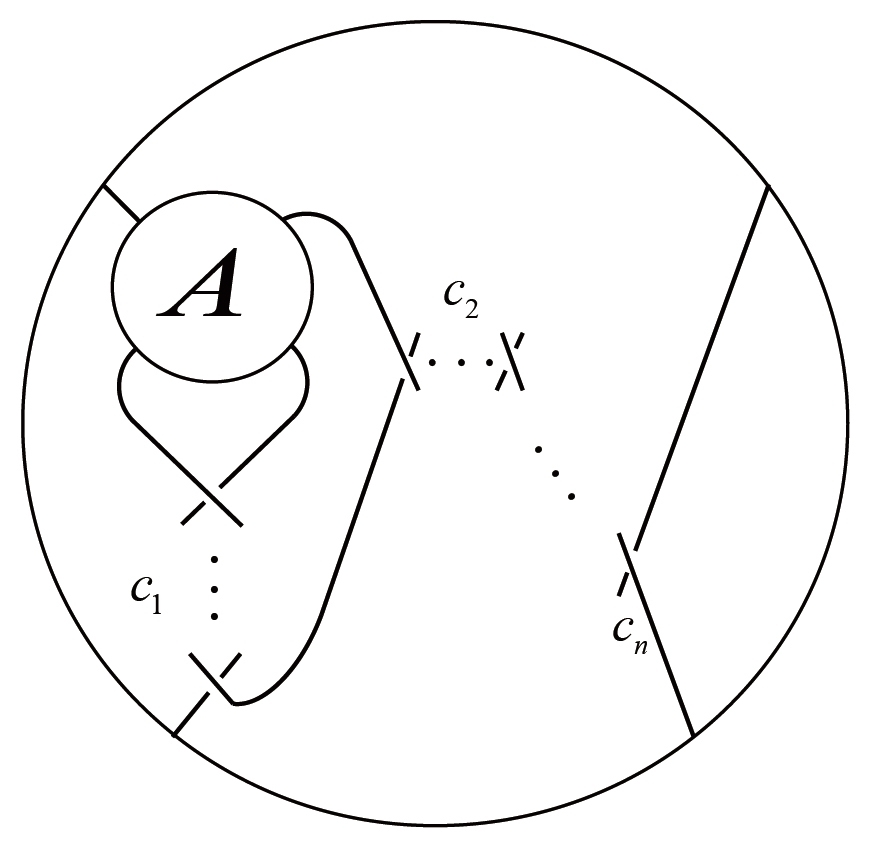}}\quad\qquad\qquad
\subfigure[$A \circ (c_1, c_2, \dots, c_n)$, when $n$ is odd.]{\label{circleproduct2}
\includegraphics[width=0.345\textwidth]{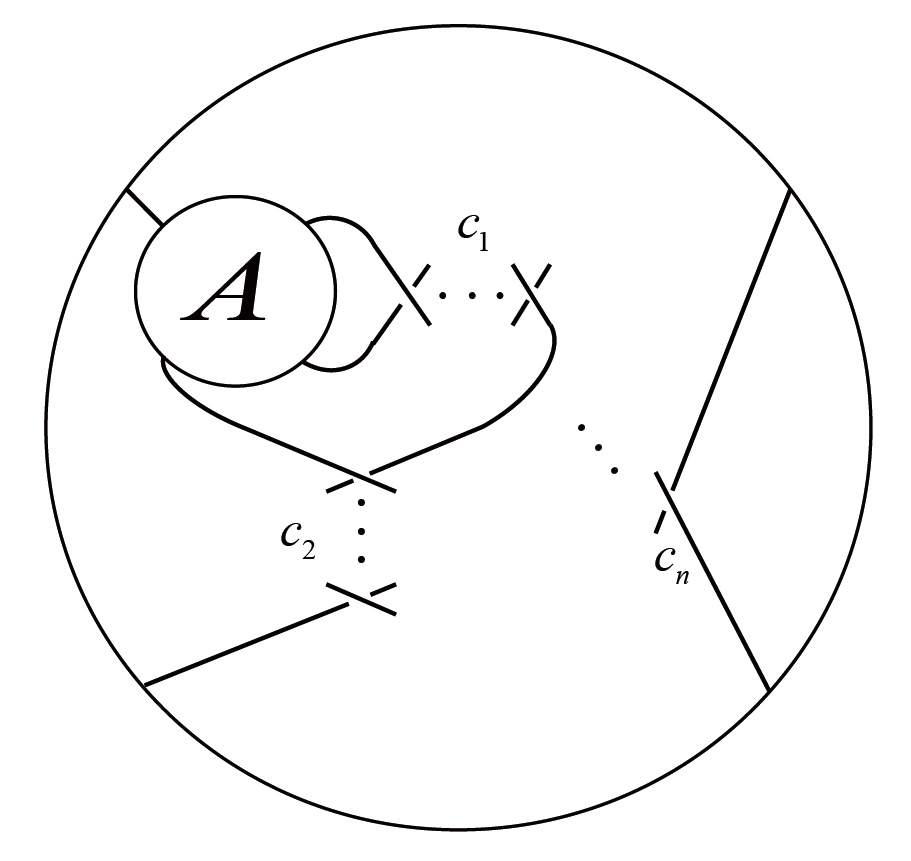}}
\caption{Circle product}\label{circleproduct}
\end{figure}

\begin{rem}
When $c_i>0$, it represents $|c_i|$ positive half twists. Conversely, $c_i<0$ represents $|c_i|$ negative half twists. Positive and negative twists are shown in the following pictures.
\end{rem}

\begin{figure}[H]
\centering
\subfigure[Positive half twist]{\label{positive half twist}
\includegraphics[width=0.21\textwidth]{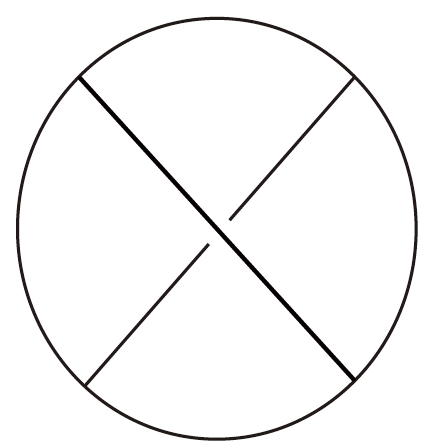}}\qquad\qquad\qquad\quad
\subfigure[Negative half twist]{\label{negative half twist}
\includegraphics[width=0.21\textwidth]{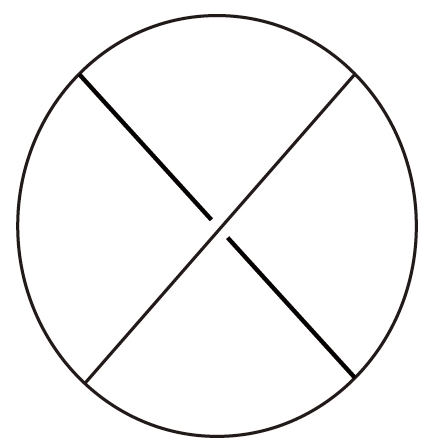}}
\caption{Half twist}
\end{figure}

\begin{defn}
A tangle $X=(B^3,t)$ is {\it rational} if it is homeomorphic to the trivial tangle $(D^2 \times I,\{x,y\} \times I)$, where $D^2$ is the unite 2-ball in $\mathbb{R}^2$ and $\{x,y\}$ are two points interior to $D^2$. See Figure \ref{trivialtangle}.
\end{defn}

Each rational tangle is isomorphic to a so-called "basic vertical tangle" which is constructed by taking the circle product of the $\infty$-tangle shown in Figure \ref{inftytangle} and an integer vector $C=(c_1,\dots, c_n)$ with $n$ even, i.e.$\infty \circ C$. The following rational tangle classification theorem tells us that rational tangles are classified, up to isomorphism, by their continued fraction:
\[\frac{\beta}{\alpha}=c_n+\frac{1}{c_{n-1}+\frac{1}{\dots +\frac{1}{c_1}}}\].

\begin{thm}[Rational Tangle Classification Theorem \cite{Conway}]
\label{Rational Tangle Classification Theorem}
There exists a 1-1 correspondence between isomorphism classes of rational tangles and the extended rational numbers $\beta / \alpha \in \mathbb{Q} \cup \{ 1/0=\infty \} $, where $\alpha \in \mathbb{N}, \,\, \beta \in \mathbb{Z}$, and $gcd(\alpha, \beta)=1$.
\end{thm}

\begin{figure}[H]
\centering
\subfigure[0-tangle]{\label{0tangle}
\includegraphics[width=0.21\textwidth]{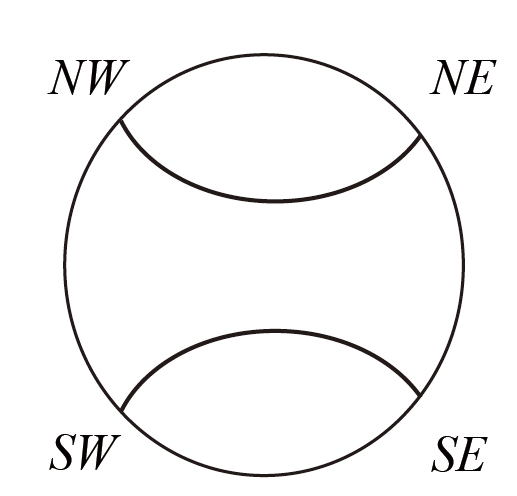}}\quad\qquad
\subfigure[$\infty$-tangle]{\label{inftytangle}
\includegraphics[width=0.22\textwidth]{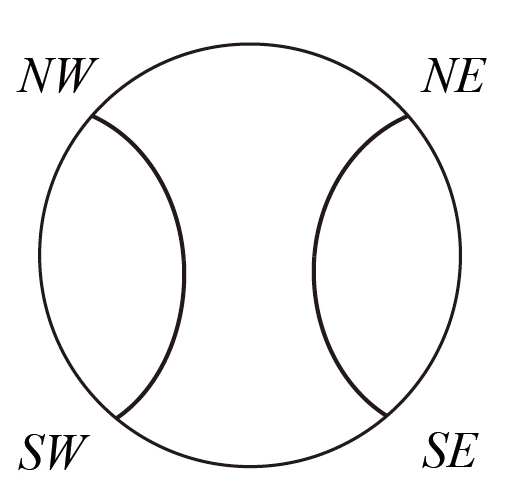}}\quad\qquad
\subfigure[The $\frac{13}{4}$-tangle]{\label{rationaltangle}
\includegraphics[width=0.22\textwidth]{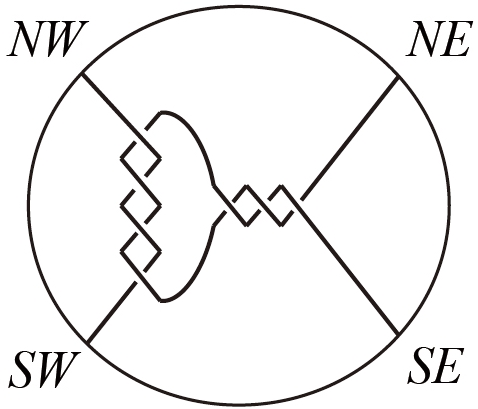}}
\caption{Rational tangles}\label{trivialtangle}
\end{figure}


A {\it 2-bridge knot/link (4-plat knot/link or rational knot/link)} is a knot/link obtained by taking numerator or denominator closure of rational tangles. More precisely, $D(\frac{\beta}{\alpha})$ gives a 2-bridge knot/link, denoted by $b(\alpha, \beta)$. $N(\frac{\beta}{\alpha})$ gives the 2-bridge knot/link $b(\beta,-\alpha)$. In fact, the numerator or denominator closure of rational tangles also produce the unknot. For convenience, the unknot is contained in the 2-bridge knot/link class in this paper, and it is denoted by $b(1,1)$.

\begin{thm}[2-Bridge Link Classification Theorem \cite{Schubert}]
The 2-bridge knot/link $b(\alpha, \beta)$ with $\alpha>0$ is equivalent to the 2-bridge knot/link $b(\alpha', \beta')$ with $\alpha'>0$ if and only if $\alpha=\alpha'$ and $\beta^{\pm1}\equiv \beta' (mod \,\alpha)$.
\end{thm}

Adding two rational tangles and then taking the numerator closure also gives a 2-bridge knot/link. See the following lemma.

\begin{lem}{\rm \cite{Sumners}}
\label{rational tangle calculus}
Given two rational tangles $X_1=\beta_1 / \alpha_1$ and $X_2=\beta_2 / \alpha_2$, then $N(X_1+X_2)$ is the 2-bridge knot/link $b(\alpha_1 \beta_2 + \alpha_2 \beta_1, \alpha_1 \beta'_2 + \alpha'_2\beta_1)$, where $\beta_2 \alpha'_2 - \alpha_2 \beta'_2=1$.
\end{lem}

\begin{defn}
 The {\it Montesinos tangle} $\left( \frac{\beta_1}{\alpha_1}, \dots, \frac{\beta_i}{\alpha_i}, \dots, \frac{\beta_n}{\alpha_n} \right)$, where $n \geq 2$, is obtained from adding $n$ rational tangles $\beta_i / \alpha_i$ together, that is $\left(\frac{\beta_1}{\alpha_1}, \dots, \frac{\beta_i}{\alpha_i}, \dots, \frac{\beta_n}{\alpha_n} \right)=\frac{\beta_1}{\alpha_1}+ \dots +\frac{\beta_i}{\alpha_i}+ \dots +\frac{\beta_n}{\alpha_n}$. See Figure \ref{Montesinostangle}.
\end{defn}

\begin{figure}[H]
\centering
\includegraphics[width=2.2in]{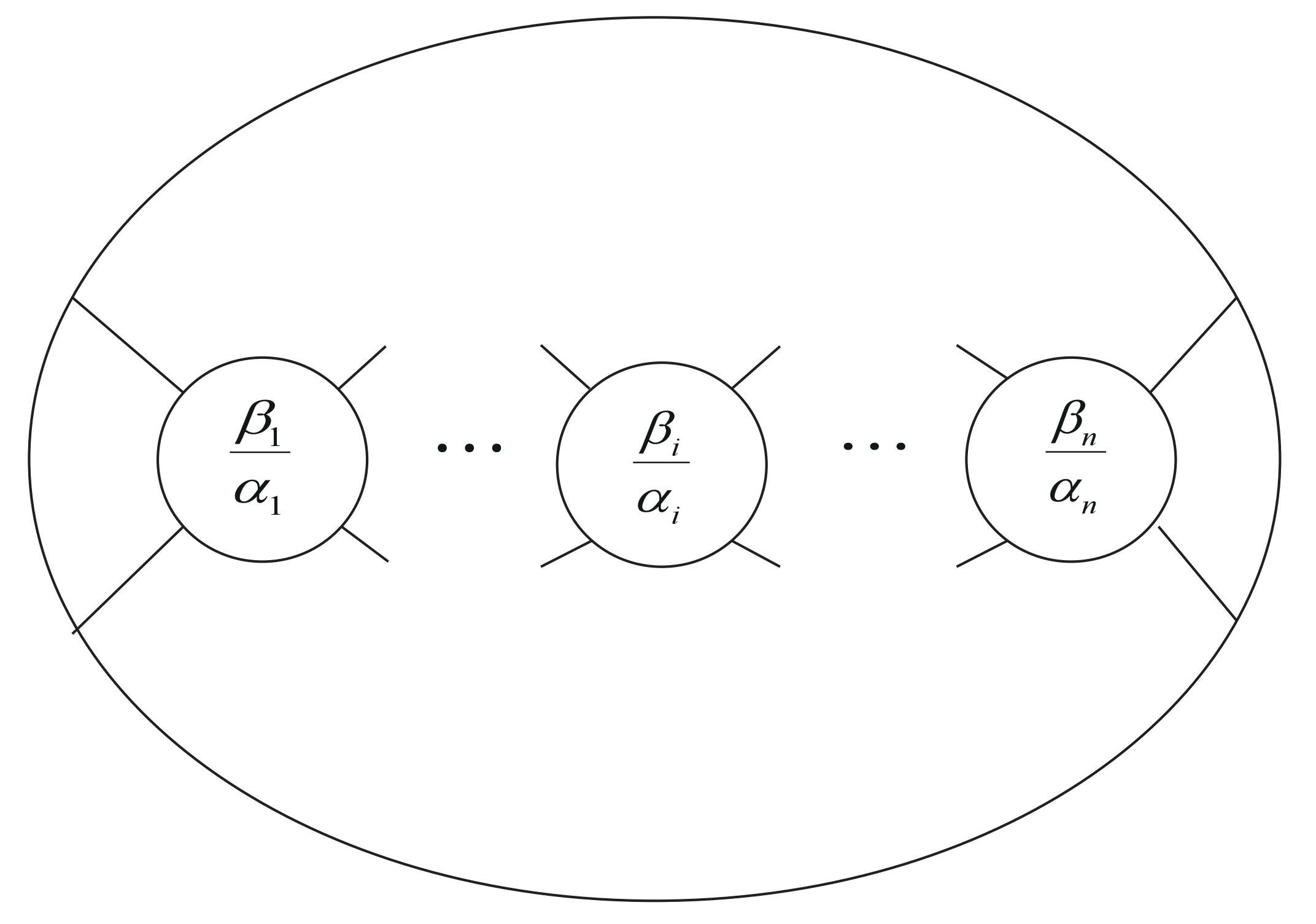}
\caption{Montesinos tangle \label{Montesinostangle}}
\end{figure}

\begin{defn}
A {\it Generalized Montesinos tangle} is a tangle obtained by taking the circle product of a Montesinos tangle $(A_1, \dots, A_n)$ and an integer entry vector $C=(c_1, c_2, \dots, c_m)$, where $A_i$ is a rational tangle for $i=1,\dots,n$, and $m \in \mathbb{Z}^{+}$, denoted by $(A_1+ \dots +A_n) \circ C$. See Figure \ref{generalizedmontesinostangle}.
\end{defn}

\begin{figure}[H]
\centering
\includegraphics[width=2.3in]{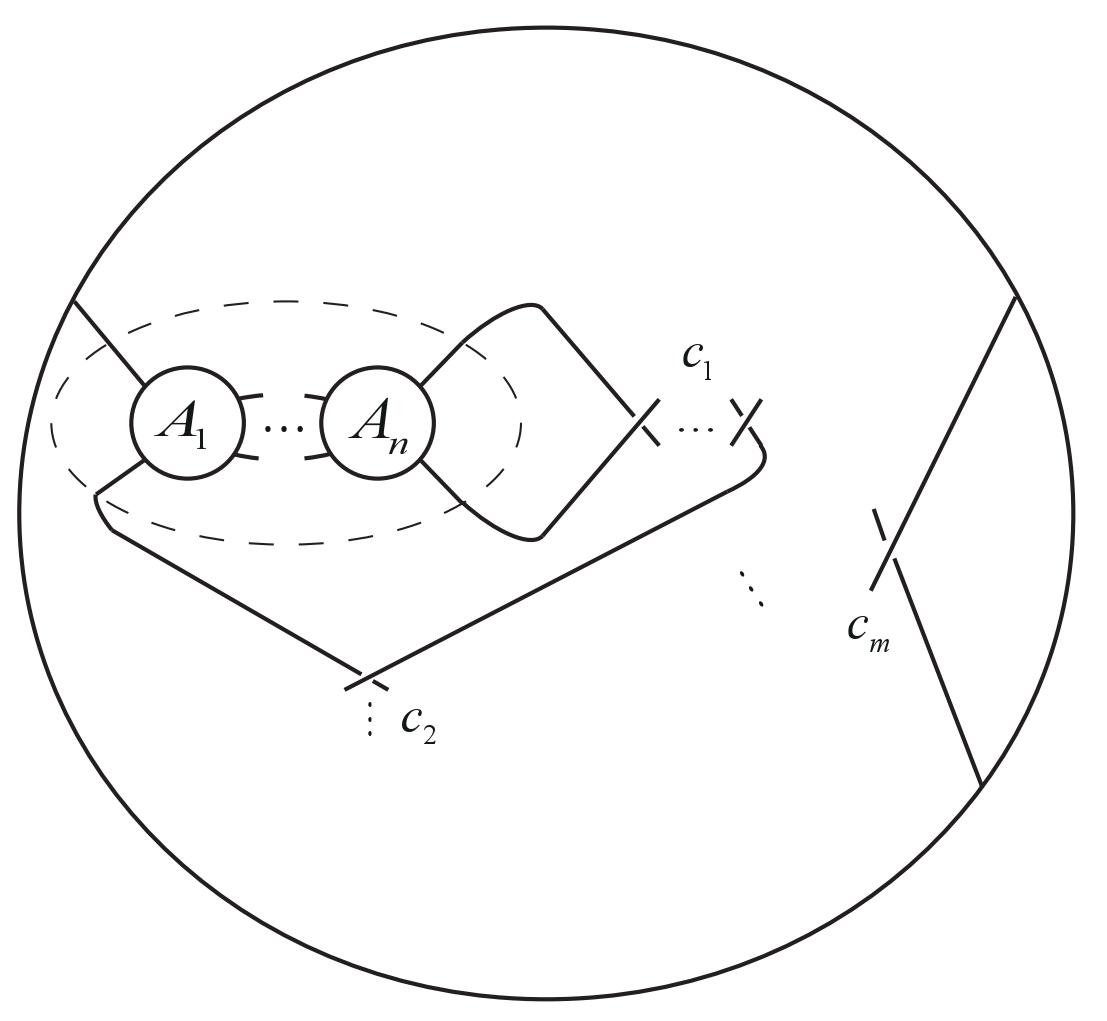}
\caption{Generalized Montesinos tangle \label{generalizedmontesinostangle}}
\end{figure}

\begin{defn}
An {\it algebraic tangle} is a tangle obtained by preforming the operations of tangle addition and multiplication on rational tangles.
\end{defn}

\begin{rem}
Algebraic tangle includes all of rational tangles, Montesinos tangles and generalized Montesinos tangles.
\end{rem}

\begin{defn}
A tangle $X=(B^3,t)$ is {\it locally knotted} if there exists a local knot in one of the strands. More precisely, there exists a 2-sphere in $X$ intersecting $t$ transversely in 2 points, such that the 3-ball it bounds in $X$ meets $t$ in a knotted arc.
\end{defn}

Now we introduce a special type of tangle which is not necessarily in a 3-ball. Following Bonahon and Siebenmann \cite{Bonahon}, we call it a Montesinos pair instead of a tangle, since there is a specifying definition of a Montesinos tangle. Note that the previous Montesinos tangles and generalized Montesinos tangles are included in Montesinos pair.

\begin{defn}
A {\it Montesinos pair} is homeomorphic to a tangle constructed from a tangle of the type (a) or (b) shown in Figure \ref{Montesinospair} by plugging some of the holes with rational tangles.
\begin{enumerate}[(1)]
\item If a Montesinos pair $M$ is built from the tangle of the type (a) with rational tangle $\frac{\alpha_i}{\beta_i}$ plugging in from left to right, and with $k$ boundary components(i.e.$k$ holes with no rational tangle plugging in), then we denote $M$ as $M=(0,k;\frac{\alpha_1}{\beta_1}, \dots, \frac{\alpha_N}{\beta_N})$. If there is no rational tangle plugging in some of the holes, then we fill in the symbol $\varnothing$ instead of a rational number.
\item If a Montesinos pair $M$ is built from the tangle of the type (b) with rational tangle $\frac{\alpha_i}{\beta_i}$ plugging in from left to right, and with $k$ boundary components, then we denote $M$ as $M=(-1,k;\frac{\alpha_1}{\beta_1}, \dots, \frac{\alpha_N}{\beta_N})$. If there is no rational tangle plugging in some of the holes, then we fill in the symbol $\varnothing$ instead of a rational number.
\end{enumerate}
\end{defn}

\begin{figure}[H]
\centering
\subfigure[]{\label{Montesinospaira}
\includegraphics[width=0.38\textwidth]{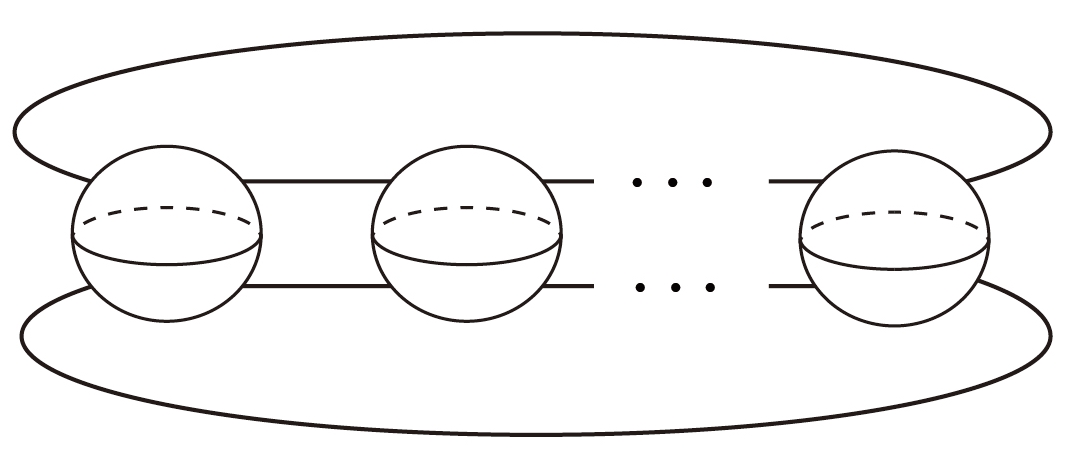}}\qquad
\subfigure[]{\label{Montesinospairb}
\includegraphics[width=0.4\textwidth]{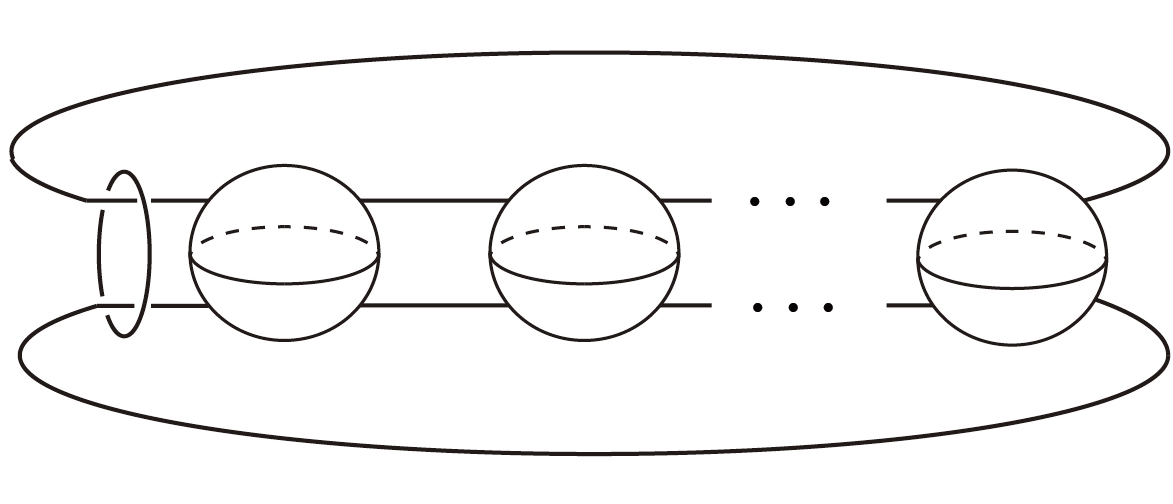}}\qquad
\subfigure[The ring tangle]{\label{ringtangle}
\includegraphics[width=0.18\textwidth]{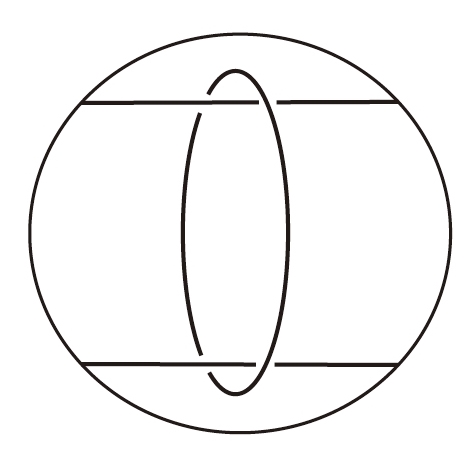}}
\caption{Montesinos pairs built from the type (a) (resp.(b)) contain no ring tangle (resp. one ring tangle). The ring tangle is a tangle shown in (c).}\label{Montesinospair}
\end{figure}

\begin{defn}
A {\it Conway sphere} in a tangle $(M,t)$ is a 2-sphere in int$M$ which intersects $t$ transversely in 4 points.
\end{defn}

\subsection{Double branched covers}
\label{section of Double branched covers}

The key point to solve tangle equations is considering their corresponding double branched covers. When we perform tangle addition and then take the numerator closure on two tangles $O$ and $X$, it gives a knot $K$. Lifting to their double branched covers, these operations induce a gluing of the boundaries of their respective double branched covers, $\widetilde{O}$ and $\widetilde{X}$. It yields a 3-manifold $\widetilde{K}$, the double cover of $S^3$ branched over $K$:\\
\centerline{$N(O+X)=K\Longrightarrow \widetilde{O} \cup_h \widetilde{X}=\widetilde{K}$}
where $h: \partial \widetilde{O} \rightarrow \partial \widetilde{X}$ is the gluing map. In particular, when $X$ is a rational tangle, this tangle equation corresponds to a Dehn filling on $\widetilde{O}$.

A {\it slant} on $(S^2, P=4 \, points)$ is the isotopy class of essential simple closed curves in $S^2-P$. Let $f: T^2 \rightarrow S^2$ be the double covering map branched over $P$ (It is well known that the double cover of $S^2$ branched over $P$ is a torus). For each slant $c$, $f^{-1}(c)$ are two parallel simple closed curves in $T^2$, and we denote an arbitrary one as $\widetilde{c}$. Let $m$, $l$ be the slants on $(S^2, P)$ shown as the following figure:

\begin{figure}[H]
\centering
\includegraphics[width=1.1in]{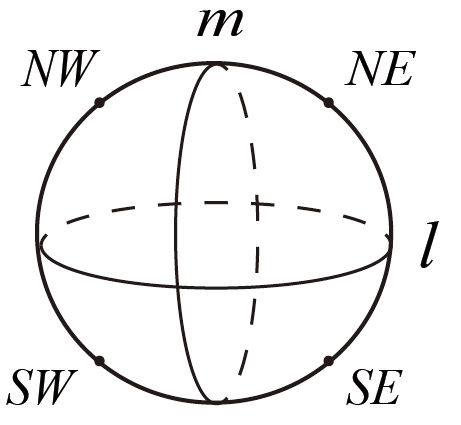}
\end{figure}

Orient $\widetilde{m}$ and $\widetilde{l}$ so that $\widetilde{m} \cdot \widetilde{l}=1$ with respect to the orientation of the torus. Then $[\widetilde{m}]$ and $[\widetilde{l}]$ form a basis of $H_1(T^2)$. In fact, there are bijections \{slants on $(S^2,P)$\} $\leftrightarrow$ \{slopes on $T^2$\} $\leftrightarrow$ \{$\mathbb{Q} \cup \infty$\}. For example, the slant $l$ corresponds to the rational number $0$, and $m$ corresponds to $\infty$. For a $\frac{q}{p}$ slant on $(S^2,P)$, we can construct it by firstly taking $p$ parallel copies of $l$ and then performing a $\frac{\pi q}{p}$ twist along $m$. This $\frac{q}{p}$ slant lifts to a $p[\widetilde{l}]+q[\widetilde{m}]$ slope on $T^2$.

For each rational tangle, there exists a properly embedded disk that separates the tangle into two 3-balls each containing an unknotted arc. We call this disk a {\it meridional disk}, and it can be shown that meridian disk is unique up to isotopy. The boundary of the meridional disk for a rational tangle is a slant on the boundary of the tangle, called the {\it meridian} of this rational tangle. Also the meridian of a rational tangle is unique up to isotopy. In fact the fraction corresponding to the meridian for a given rational tangle coincides with its continued fraction which we use to classify rational tangles in Theorem \ref{Rational Tangle Classification Theorem}. Also It is well known that the double branched cover of a rational tangle is a solid torus. So there is a 1-1 correspondence between plugging one boundary component of a tangle in the rational tangle $q / p$ and the $q / p$ Dehn filling (under $[\widetilde{m}]$ and $[\widetilde{l}]$ basis) of the corresponding torus boundary component in the double branched cover of the tangle.

Now we give some useful results about double branched covers.\\
\begin{align*}
&X: a \,\, tangle \,\, (B^3,t)           &&\widetilde{X}: the \,\, double \,\, branched \,\, cover \,\, of \,\, B^3 \,\, with \,\, branched \,\, set \,\, t.\\
&K: a \,\, link \,\, in \,\, S^3 \,\,    &&\widetilde{K}: the \,\, double \,\, branched \,\, cover \,\, of \,\, S^3 \,\, with \,\, branched \,\, set \,\, K.\\
&Y: a \,\, tangle \,\, (M,t)     &&\widetilde{Y}: the \,\, double \,\, branched \,\, cover \,\, of \,\, M \,\, with \,\, branched \,\, set \,\, t.
\end{align*}

\begin{itemize}
  \item A tangle $X$ is rational if and only if $\widetilde{X}$ is a solid torus.
  \item A knot/link $K$ is a 2-bridge knot/link if and only if $\widetilde{K}$ is a lens space. The double branched cover of the 2-bridge knot/link $b(p,q)$ is the lens space $L(p,q)$.
  \item The double branched cover of $K_1 \# K_2$ is $\widetilde{K_1} \# \widetilde{K_2}$, where $K_i$ is a knot/link for $i=1,2$ \cite{Buck}.
  \item The double branched cover of a locally knotted tangle is a reducible manifold.
  \item The double branched cover of an algebraic tangle is a graph manifold with one torus boundary (a graph manifold is a 3-manifold which is obtained by gluing some circle bundles) \cite{Bonahon}.
  \item The double branched cover of a Montesinos pair $Y=(M,t)$ is a generalized Seifert fiber space with orbit surface of genus 0 or -1. If $Y=(0,k;\frac{\beta_1}{\alpha_1}, \dots, \frac{\beta_N}{\alpha_N})$, then $\widetilde{Y}=M(0,k;(\alpha_1,\beta_1), \dots, (\alpha_N,\beta_N))$. If $Y=(-1,k;\frac{\beta_1}{\alpha_1}, \dots, \frac{\beta_N}{\alpha_N})$, then $\widetilde{Y}=M(-1,k;(\alpha_1,\beta_1), \dots, (\alpha_N,\beta_N))$ \cite{Bonahon}.
  \item A Seifert fiber space over a disk with $n$ exceptional fibers is the double branched cover of a tangle, then the tangle is a generalized Montesinos tangle with $n$ rational tangle summands \cite{C.ernst}.
\end{itemize}

Generalized Seifert fiber space is defined similarly to Seifert fiber space, and it comprises all true Seifert fiber spaces. The only difference is that the exceptional fiber of a generalized Seifert fiber space may have the coefficient $(0,1)$. For more details, please see \cite{JN}.

Each double branched covering map is induced by an involution. We call the involution of a torus with 4 fixed points the {\it standard involution}. As mentioned above, $M(0,k;(\alpha_1,\beta_1), \dots, (\alpha_N,\beta_N))$ is the double branched cover of the Montesinos pair $Y=(0,k;\frac{\beta_1}{\alpha_1}, \dots, \frac{\beta_N}{\alpha_N})$. We call the non-trivial covering transformation corresponding to this double branched cover the {\it standard involution} on $M(0,k;(\alpha_1,\beta_1), \dots, (\alpha_N,\beta_N))$.

\section{Solving tangle equations}
\label{Solving tangle equations}

The main goal of this paper is to solve the following tangle equations:

\begin{align}
N(O+X_1)&=b_1          \label{eq1}\\
N(O+X_2)&=b_2 \# b_3,  \label{eq2}
\end{align}

where $X_1$ and $X_2$ are rational tangles, and $b_i$ is a 2-bridge link, for $i=1,2,3$, with $b_2$ and $b_3$ nontrivial.

Given the above tangle equations, we also say that the knot/link $b_2 \# b_3$ can be obtained from $b_1$ by an $(X_1, X_2)$-move. Let $X_i$ and $X'_i$ be rational tangles for $i=1,2$. An $(X_1,X_2)$-move is {\it equivalent} to an $(X'_1,X'_2)$-move if for any two knots/links $K_1$ and $K_2$, there exists a tangle $O$ satisfying tangle equations $N(O+X_1)=K_1$ and $N(O+X_2)=K_2$, if and only if there exists a tangle $O'$ such that $N(O'+X'_1)=K_1$ and $N(O'+X'_2)=K_2$. As discussed in \cite{Darcy}, any solution $(O,X_1,X_2)$ of the above tangle equations is equivalent to a solution $(O',X'_1=0,X'_2)$. If one solution $(O,X_1,X_2)$ satisfying $X_1=0$-tangle is given, there is a standard algorithm to give all equivalent solutions. For more details, please see \cite{Darcy}. So we only need to solve the above tangle equations assuming $X_1=0$-tangle.

Lifting to the double branched covers, the system of tangle equations can be translated to
\begin{align}
&\widetilde{O}(\theta)= the \,\, lens \,\, space \quad \widetilde{b_1}                                                        \label{eq3}\\
&\widetilde{O}(\eta)= the \,\, connnected \,\, sum \,\, of \,\, two \,\, lens \,\, spaces \,\, \widetilde{b_2} \,\,and \,\, \widetilde{b_3} \quad \widetilde{b_2} \# \widetilde{b_3},  \label{eq4}
\end{align}

where $\widetilde{O}$ (resp.$\widetilde{b_i}$) denotes the double branched cover of $O$(resp.$b_i$), and $\theta$ (resp.$\eta$) is the induced Dehn filling slope by adding rational tangle $X_1$(resp.$X_2$) since $\widetilde{X_i}$ is a solid torus. Therefore, the problem turns out to be finding knots in the lens space $\widetilde{b_1}$ which admit a surgery to a connected sum of two lens spaces. Here $\widetilde{X_1}$ is a tubular neighborhood of the knot in $\widetilde{b_1}$, and $\widetilde{O}$ is the complement of the knot.

In fact, Kenneth L.Baker gave a lens space version of cabling conjecture.

\begin{conj}(The Lens Space Cabling Conjecture \cite{Baker}) Assume a knot $K$ in a lens space $L$ admits a surgery to a non-prime 3-manifold $Y$. If $K$ is hyperbolic, then $Y=L(r,1)\#L(s,1)$. Otherwise either $K$ is a torus knot, a Klein bottle knot, or a cabled knot and the surgery is along the boundary slope of an essential annulus in the exterior of $K$, or $K$ is contained in a ball.
\end{conj}

The non-hyperbolic case has been proved by Baker in \cite{Baker}, although Fyodor Gainullin \cite{Fyodor Gainullin} constructed a counterexample to this conjecture for the hyperbolic case. So we can use this conjecture to solve the equations under the assumption that $\widetilde{O}$ is not hyperbolic, i.e.$O$ is not $\pi$-hyperbolic.

\begin{defn}
A tangle is {\it $\pi$-hyperbolic} if its double branched cover admits a hyperbolic structure and the covering translation is an isometry.
\end{defn}

According to the lens space cabling conjecture, there are three cases except hyperbolic case:
\begin{enumerate}[(1)]
  \item when $\widetilde{O}$ is reducible, $\widetilde{O}$ is the complement of a knot contained in a ball in a non-trivial lens space $\widetilde{b_1}$.
  \item when $\widetilde{O}$ is a Seifert fiber space, $\widetilde{O}$ is the complement of a torus knot or a Klein bottle knot in $\widetilde{b_1}$.
  \item when $\widetilde{O}$ is irreducible toroidal but not Seifert fibered, $\widetilde{O}$ is the complement of a cabled knot in $\widetilde{b_1}$.
\end{enumerate}

In fact, (2)(3) make up the case that $\widetilde{O}$ is irreducible but not hyperbolic.

\subsection{Case \uppercase\expandafter{\romannumeral1}: $\widetilde{O}$ is reducible}

If $\widetilde{O}$ is reducible, then $\widetilde{O}$ contains essential 2-spheres. $\widetilde{O}$ is the double branched cover of a tangle $O=(B^3,t)$ with an associated involution $\sigma$, namely $\widetilde{O}/ \sigma=O$. The essential 2-spheres in $\widetilde{O}$ either do not intersect the fixed points of the involution $\sigma$, denoted by $fix(\sigma)$, or intersect $fix(\sigma)$ transversely in 2 points. The essential 2-spheres which do not intersect $fix(\sigma)$ are mapped to essential 2-spheres in the complement of $O$(i.e.$B^3-t$) by the covering map induced by $\sigma$. The essential 2-spheres in $O$ remain essential after adding the tangle $X_1$ and taking the numerator closure. In other words, $N(O+X_1)=b_1$ is split. Since $b_1$ is a 2-bridge link, $b_1$ is split if and only if $b_1=b(0,1)$ (i.e.the unlink). So there is at most one essential 2-sphere in $O$ which splits from $O$ the unknot. For any essential 2-sphere $S$ which meets $fix(\sigma)$ in 2 points, we can assume that it is invariant under $\sigma$, then $S/\sigma$ is a 2-sphere intersecting $t$ in 2 points and the 3-ball it bounds in $O$ meets $t$ in a knotted arc. There is at most one local knot in $O$ since $b_1$ is a 2-bridge link which is prime. In addition, it is impossible for $O$ to contain both locally knotted arc and a splittable unknot.

The locally knotted case has been discussed in Buck and Mauricio's paper \cite{Buck}. In \cite{Buck}, by excising the knotted arc in $O$,  the tangle equations \ref{eq1} and \ref{eq2} reduce to the following equations:

\begin{align}
N(O'+X_1)&=b(1,1) \quad (the \,\, unknot) \label{eq12}\\
N(O'+X_2)&=b_3,                           \label{eq13}
\end{align}

where $O'$ is the tangle we obtain after excising the knotted arc in $O$. In the tangle equations \ref{eq1} and \ref{eq2}, there exists $i=2$ or $3$ such that $b_i=b_1$, both of which contain the information of the knotted arc. Otherwise there does not exist any locally knotted solution.  Without loss of generality, we assume $i=2$. Once we get a solution $O'$, then we can recover the original $O$ by letting $O=O' \# b_1$.

When $O$ contains a splittable unknot, both $b_1$ and $b_2 \# b_3$ contain one. It implies one of $b_2$ and $b_3$ is $b(0,1)$, without loss of generality, $b_2=b(0,1)$. Also we can remove the splittable unknot, then the tangle equations \ref{eq1} and \ref{eq2} reduce to the system of equations above. Once we get a solution $O'$, then we can recover the original $O$ by putting the splittable unknot back.

Lifting to the double branched covers, we have
\begin{align*}
&\widetilde{O'}(\theta)=S^3\\
&\widetilde{O'}(\eta)=the\,lens\,space \,\, \widetilde{b_3},
\end{align*}

Therefore $\widetilde{O'}$ is the complement of a knot $K$ in $S^3$. The problem turns out to be finding the knot $K$ in $S^3$ which admits lens space surgeries.  This is still an open question. But for some knot, this question is well-understood. If $K$ is the unknot, then $\widetilde{O'}$ is a solid torus, which means $O'$ is a rational tangle, and this case can be solved by using Lemma \ref{rational tangle calculus}. If $K$ is a torus knot, then $\widetilde{O'}$ is a Seifert fiber space over a disk, which means $O'$ is a generalized Montesinos tangle, and this case is solved in \cite{DarcySumners} and \cite{Darcy}. If $K$ is neither the unknot nor a torus knot, then this case can be related to Berge knots which is a family of knots constructed by John Berge with lens space surgeries. Berge conjecture states that Berge knots contain all the knots in $S^3$ admitting lens space surgeries. Besides, Joshua Greene \cite{Greene} proved that the lens spaces obtained by doing surgeries along a knot in $S^3$ are precisely the lens spaces obtained by doing surgeries along Berge knots. It implies that only if $b_3$ satisfies that $\widetilde{b_3}$ is a lens space arising from surgeries along Berge knots, then the solution of the system of tangle equations \ref{eq12} and \ref{eq13} exists, and equivalently the solution of the system of tangle equations \ref{eq1} and \ref{eq2} exists when $\widetilde{O}$ is reducible. Meanwhile, Baker's papers\cite{Baker1}\cite{Baker2} give surgery descriptions of Berge knots on some chain links and also give tangle descriptions of Berge knots, which can help us to find some solutions in this case.

\subsection{Case \uppercase\expandafter{\romannumeral2}: $\widetilde{O}$ is irreducible but not hyperbolic}

Here, $O$ can not be a rational tangle, for otherwise $N(O+X_2)$ is a 2-bridge link by Lemma \ref{rational tangle calculus}, which is not a composite link. Then $\widetilde{O}$ is not a solid torus, thus $\widetilde{O}$ is irreducible and $\partial$-irreducible. If $\widetilde{O}$ is non-hyperbolic, then this case consists of two cases: (1)$\widetilde{O}$ is irreducible toroidal but not Seifert fibered and (2)$\widetilde{O}$ is a Seifert fiber space.

Before our analysis, we will give some useful results of Gordon. Here are some notation we will use. Let $J$ be a knot in int $S^1 \times D^2$ with winding number $w \geq 0$. Let $\alpha=[S^1 \times *] \in H_1(S^1 \times \partial D^2)$, $* \in \partial D^2$, $\beta=[* \times \partial D^2] \in H_1(S^1 \times \partial D^2)$, $* \in S^1$. Let $Y=S^1 \times D^2 \setminus N(J)$, then $H_1(Y)=\mathbb{Z}\alpha \bigoplus \mathbb{Z} \mu$ where $\mu$ is the class of a meridian of $N(J)$. There is a homeomorphism $h: S^1 \times D^2 \rightarrow N(J)$ such that $[h(S^1 \times *)]=w\alpha \in H_1(Y)$, $* \in \partial D^2$. Then $\lambda=[h(S^1 \times *)]$ and $\mu$ is a longitude-meridian basis for $J$. We use $J(r)$ to denote performing $r=m/n$ Dehn surgery on $S^1 \times D^2$ along $J$.

\begin{lem}{\rm \cite{Gordon}}
\label{Gordon1}
Let $J$ and $r$ be defined as above. Then the kernel of $H_1(\partial J(r)) \rightarrow H_1(J(r))$ is the cyclic group generated by
\begin{align*}
\displaystyle
\left\{\begin{matrix}
&\frac{nw^2}{gcd(w,m)}\alpha + \frac{m}{gcd(w,m)}\beta,  & if \,\, w \neq 0;  \\
&\beta                                             & if \,\, w=0.
\end{matrix}\right.
\end{align*}
\end{lem}

\begin{lem}{\rm \cite{Gordon}}
\label{Gordon2}
Let $J$ and $r$ be defined as above, and let $J$ be a $(p,q)$-torus knot with $p \geq 2$.
\[J(r) \cong
\left\{\begin{matrix}
&S^1 \times D^2 \# L(p,q)  &if& \, &r=pq\,& \\
&S^1 \times D^2            &if& \, &r=m/n& \,\, and \quad m=npq \pm 1
\end{matrix}\right.\]
and otherwise is a Seifert fiber space with incompressible boundary.
\end{lem}

The following are some useful results about lens space and Seifert fiber space. Let $T_i$ be a solid torus with a meridian $M_i$ and a longitude $L_i$, for $i=1,2$.

\begin{defn}
The lens space $L(a,b)=T_1 \cup_h T_2$ where $h:\partial T_2 \rightarrow \partial T_1$ is an orientation-preserving homeomorphism and $h(M_2)=aL_1+bM_1$ with $a, b \in \mathbb{Z}$ and $gcd(a,b)=1$.
\end{defn}

\begin{lem}{\rm \cite{JN}}
\label{Lens space SFS1}
\mbox{}\\
(1) If $L(a,b)$ is a generalized Seifert fiber space with orientable orbit surface, then there is a fiber preserving homeomorphism such that:
\[L(a,b) \cong M(0,0;(\alpha_1,\beta_1), (\alpha_2,\beta_2)) \cong T_1 \cup T_2\]
where
$a=det
\begin{pmatrix}
\alpha_1 &\alpha_2 \\
-\beta_1&\beta_2
\end{pmatrix}$,
$b=-det
\begin{pmatrix}
\alpha_1 &\alpha'_2 \\
-\beta_1&\beta'_2
\end{pmatrix}$ and
$det
\begin{pmatrix}
\alpha_2 &\alpha'_2 \\
\beta_2&\beta'_2
\end{pmatrix}=1$\\
(2) If $L(a,b)$ is a generalized Seifert fiber space with non-orientable orbit surface, then there is a fiber preserving homeomorphism such that
\[L(a,b) \cong M(-1,0; (\alpha,\pm 1)) \cong T_1 \cup_g (S^1 \widetilde{\times} M\ddot{o}bius \,\, band) \]
In this case $L(a,b) \cong L(4\alpha, \pm 1 -2\alpha)$.
\end{lem}

\begin{cor}{\rm \cite{DarcySumners}}
\label{Lens space SFS2}
If $L(a,b)=T_1 \cup T_2$ and $T_1$ is fibered by $H \cong pL_1+qM_1$, then $L(a,b) \cong M(0,0;(p,-e), (aq-bp,ad-be))$ where $pd-qe=1$. If $L(a,b)=T_1 \cup (S^1 \widetilde{\times} M\ddot{o}bius \,\, band)$ and $T_1$ is fibered by $H \cong pL_1+qM_1$, then $L(a,b) \cong M(-1,0;(p, \pm 1))$ and $q \cong \pm 1 \,mod \, p$.
\end{cor}

\begin{lem}{\rm \cite{DarcySumners}}
\label{Lens space SFS3}
If $\mathcal{A}=Y(s/t)$ where $Y=L(a,b) \setminus N(T_{p,q})$, $L(a,b)=T_1 \cup T_2$, and $T_{p,q}$ is a $(p,q)$-torus knot in $T_1$, then $\mathcal{A}=M(0,0;(p,-e), (aq-bp, ad-be), (s-tpq,t))$, where $pd-qe=1$.
\end{lem}

\begin{lem}{\rm \cite{DarcySumners}}
\label{Lens space SFS4}
If $\mathcal{A}=Y(s/t)$ where $Y=L(a,b) \setminus N(T_{p,q})$, $L(a,b)=T_1 \cup_g (S^1 \widetilde{\times} M\ddot{o}bius \,\, band)$, and $T_{p,q}$ is a $(p,q)$-torus knot in $T_1$, then $\mathcal{A}=M(-1,0;(p,\pm 1), (s-tpq,t))$.
\end{lem}


\subsubsection{$\widetilde{O}$ is irreducible toroidal but not Seifert fibered}

According to Baker's lens space cabling conjecture, $\widetilde{O}$ is the complement of a cabled knot in the lens space $\widetilde{b_1}$ if $\widetilde{O}$ is irreducible toroidal but not Seifert fibered, and to obtain a non-prime 3-manifold, the surgery is along the boundary slope of an essential annulus in $\widetilde{O}$.

Assume that $b_1$ is the 2-bridge link $b(a,b)$ for a pair of relatively prime integers $(a,b)$ satisfying $0<\frac{b}{a}\leq 1 (or \,\, \frac{b}{a}=\frac{1}{0}=\infty)$, then $\widetilde{b_1}=L(a,b)$. Let $K$ be the cabled knot lying in the lens space $L(a,b)$. Then there exists a knot $K'$ in $L(a,b)$ and a homeomorphism $f:S^1 \times D^2 \rightarrow N(K')$, such that $K=f(J)$ where $J=T_{p,q}$ is a $(p,q)$-torus knot in $S^1 \times D^2$ for some $p>1$, $q \in \mathbb{Z}$. $\widetilde{O}=L(a,b)\setminus N(K)$. We choose $\alpha'=\left(  f|_{S^1 \times \partial D^2}  \right)_* (\alpha)$ and $\beta'=\left( f|_{S^1 \times \partial D^2}  \right)_* (\beta)$ as a longitude-meridian basis for $\partial N(K')$. $\lambda'=\left( f|_{\partial(N(J))} \right)_*(\lambda)$, $\mu'=\left(  f|_{\partial(N(J))}  \right)_*(\mu)$ is a longitude-meridian basis for $K$. Actually, $T=f(S^1 \times \partial D^2)$ is an incompressible torus in $\widetilde{O}$. If $T$ is compressible, then $T$ either bounds a solid torus or lies in a ball in $\widetilde{O}$ since $\widetilde{O}$ is irreducible. The latter case can be excluded easily. If $T$ bounds a solid torus $V$ in $\widetilde{O}$, then $\widetilde{O}=V \cup (S^1 \times D^2 \setminus  N(T_{p,q}))$ is atoroidal, so contradicts our assumption. Splitting $\widetilde{O}$ along the incompressible torus $T$, we obtain two pieces, a cable space $C_{p,q}=S^1 \times D^2 \setminus  N(T_{p,q})$ and another piece $L(a,b) \setminus N(K')$ denoted by $M$. Now we prove that $M$ is a Seifert fiber space.

\begin{prop}
\label{Prop1}
$M$ is a Seifert fiber space over a disk with exact 2 exceptional fibers.
\end{prop}

\begin{proof}
$\infty$-surgery along $K$, i.e.not doing surgery, produces the lens space $L(a,b)$, which also means $\beta'$-Dehn filling (i.e.$\infty$-Dehn filling) on $M$ produces $L(a,b)$. The boundary slope of an essential annulus in $\widetilde{O}$ is actually the slope $r=pq$ along $K$. Performing $pq$-surgery along $K$, by Lemma \ref{Gordon2}, we have that $N(K')$ transforms into $S^1 \times D^2 \# L(p,q)$. According to Lemma \ref{Gordon1}, $p \alpha'+ q \beta'$ bounds a disk in $S^1 \times D^2 \# L(p,q)$. It means that $pq$-surgery along $K$ produces $M(q/p) \# L(p,q)$. According to our equations, we expect that $M(q/p)$ is a lens space and $L(p,q) \neq S^3$ (i.e.$p \neq 1$). So now we have $M(1/0)$ and $M(q/p)$ are two lens spaces with $p \neq 1$. Obviously $\triangle (1/0,q/p)>1$, by Cyclic Surgery Theorem \cite{Culler}, $M$ must be a Seifert fiber space.

$M$ is a Seifert fiber space admitting two Dehn filings to produce lens space. By Lemma \ref{Lens space SFS1}, $M$ can be a Seifert fiber space over a disk with at most two exceptional fibers or a Seifert fiber space over a $M\ddot{o}bius \, band$ with at most one exceptional fiber. If $M$ is a Seifert fiber space over a disk, then $M$ must have exact two exceptional fibers. Because otherwise $\partial M$ is compressible in $M$, which contradicts our previous analysis. If $M$ is over a $M\ddot{o}bius \, band$ without exceptional fiber, then we can choose another fibration of $M$ such that $M$ is a Seifert fiber space over a disk with 2 exceptional fibers, since $M(-1,1;)=M(0,1;(2,1),(2,-1))$. Suppose $M=M(-1,1; (a,b))$ with $a>1$. There are two Dehn fillings with $\triangle >1$ on $M$ producing lens spaces as our analysis above. Assume that the two fillings add two fibers $(a_1,b_1)$ and $(a_2,b_2)$ to $M$ respectively ($a_i=1$ means adding an ordinary fiber; $a_i>1$ means adding an exceptional fiber). Then we obtain $M(-1,0;(a,b),(a_1,b_1))$ and $M(-1,0; (a,b),(a_2,b_2))$. To be lens spaces, by Lemma \ref{Lens space SFS1}, we must have $a_1=a_2=1$, and then $M(-1,0;(a,b),(a_1,b_1))=M(-1,0;(a,b+ab_1))$ and $M(-1,0;(a,b),(a_2,b_2))=M(-1,0;(a,b+ab_2))$ satisfying one of the four systems of equations:\\

$\left\{\begin{matrix}
b+ab_1=1\\
b+ab_2=1
\end{matrix}\right.$ \quad
$\left\{\begin{matrix}
b+ab_1=-1\\
b+ab_2=-1
\end{matrix}\right.$ \quad
$\left\{\begin{matrix}
b+ab_1=-1\\
b+ab_2=1
\end{matrix}\right.$ \quad
$\left\{\begin{matrix}
b+ab_1=1\\
b+ab_2=-1
\end{matrix}\right.$.

The first two systems of equations are impossible. Because if one holds, then $a(b_1-b_2)=0$, so $b_1=b_2$, which contradicts $\triangle >1$. The last two systems of equations are also impossible. Because if one holds, then $a(b_1-b_2)=\pm 2$ and $a>1$, thus $b_1-b_2=\pm 1$. It implies $\triangle(b_1/a_1, b_2/a_2)=1$, which also contradicts our analysis. Therefore $M$ is a Seifert fiber space over a disk with exact two exceptional fibers.

\end{proof}

Now, we have $M=L(a,b)\setminus N(K')$ is a Seifert fiber space over a disk with 2 exceptional fibers, so $K'$ is isotopic to a fiber in some generalized Seifert fibration of $L(a,b)$ over a 2-sphere. Moreover, the fiber is an ordinary fiber, for otherwise $M=L(a,b) \setminus N(K')$ has at most one exceptional fiber. In fact, we can regard $L(a,b)$ as a union of two solid tori, $T_1$ and $T_2$, and $K'$ is isotopic to a $(p_1,q_1)$-torus knot in $T_1$ for some $p_1 >1, q_1 \in \mathbb{Z}$. Because if $p_1=1$, then $M=L(a,b)\setminus N(K')$ is a solid torus instead of a Seifert fiber space with 2 exceptional fibers; if $p_1=0$, then $K$ is lying in a ball, thus either $\widetilde{O}=L(a,b)\setminus N(K)$ is reducible which contradicts our assumption, or $L(a,b)=S^3$. $L(a,b)=S^3$ is impossible, since if it is and $p_1=0$, then $M=L(a,b)\setminus N(K')=S^3\setminus N(K')$ is a solid torus. We chose a longitude-meridian basis $(\lambda_1, \mu_1)$ for $\partial N(K')$ using the same principle as choosing basis for $J$ in $S^1 \times D^2$. Now we redefine the homeomorphism $f: S^1 \times D^2 \rightarrow N(K')$ such that $(f|_{S^1 \times \partial D^2})_*(\alpha) =\lambda_1$ and $(f|_{S^1 \times \partial D^2})_*(\beta) =\mu_1$, then $K=f(J)$ where $J=T_{p,q}$ for some $p>1$, $q \in \mathbb{Z}$. So now, $K$ is a $(p,q)$-cable of $(p_1,q_1)$-torus knot lying in $T_1$ of $L(a,b)$ with $p>1$ and $p_1>1$.

\begin{prop}
\label{Prop2}
$K$ is a $(p,q)$-cable of $(p_1,q_1)$-torus knot lying in $T_1$ of $L(a,b)=T_1 \cup T_2$ with $q=pp_1q_1 \pm 1$ and $|aq_1-bp_1|>1$, where $p>1$, $p_1>1$, $0<\frac{b}{a}\leq 1 (or \,\, \frac{b}{a}=\frac{1}{0}=\infty)$, and $T_i$ is a solid torus, for $i=1,2$. When the surgery slope is $r=pq$, the manifold obtained is $L(p,q) \# L(a \pm ap_1q_1p \mp bp_1^2p, b \pm aq_1^2p \mp bp_1q_1p)$.
\end{prop}

\begin{proof}
By Lemma \ref{Gordon2} and Lemma \ref{Gordon1},
\[K(pq)=[L(a,b) \setminus N(T_{p_1,q_1})](q / p) \# L(p, q),\]
where $K(r)$ denotes $r$-surgery along $K$. By Lemma \ref{Lens space SFS3},
\[[L(a,b) \setminus N(T_{p_1,q_1})](q / p) = M(0,0;(p_1,-e),(aq_1-bp_1,ad-be),(q-pp_1q_1, p)),\]
where $p_1d-q_1e=1$. We want it to be a lens space, and that happens if and only if one of $|q-pp_1q_1|$, $|aq_1-bp_1|$ and $|p_1|$ equals 1. In fact, $|aq_1-bp_1|>1$, for the same reason as $p_1>1$ since a $(p_1,q_1)$-torus knot in $T_1$ of $L(a,b)$ is also a torus knot with winding number $|aq_1-bp_1|$ in $T_2$. So we have $|q-pp_1q_1|=1$. There are two cases.

(\romannumeral1) $q-pp_1q_1=1$
\begin{align*}
[L(a,b) \setminus N(T_{p_1,q_1})](q / p) &= M(0,0;(p_1,-e),(aq_1-bp_1,ad-be),(1,p))\\
                                         &=M(0,0;(p_1,-e),(aq_1-bp_1,ad-be+aq_1p-bp_1p)).
\end{align*} Using the formula in Lemma \ref{Lens space SFS1}, then
\begin{align*}
&M(0,0;(p_1,-e),(aq_1-bp_1,ad-be+aq_1p-bp_1p))\\
&=L(a+ap_1q_1p-bp_1^2p, b+aq_1^2p-bp_1q_1p).
\end{align*}

(\romannumeral2) $q-pp_1q_1=-1$
\begin{align*}
[L(a,b) \setminus N(T_{p_1,q_1})](q / p) &=M(0,0;(p_1,-e),(aq_1-bp_1,ad-be),(1,-p))\\
                                         &=M(0,0;(p_1,-e),(aq_1-bp_1,ad-be-aq_1p+bp_1p)).
\end{align*} Using the formula in Lemma \ref{Lens space SFS1},
\begin{align*}
&M(0,0;(p_1,-e),(aq_1-bp_1,ad-be-aq_1p+bp_1p))\\
&=L(a-ap_1q_1p+bp_1^2p, b-aq_1^2p+bp_1q_1p).
\end{align*}
\end{proof}

Before proving the main theorem, we will give a useful proposition about the mapping class group of four-times-punctured sphere denoted by $S_{0,4}$. We use $Mod$ to denote mapping class group.

\begin{prop}{\rm \cite{Farb}}
\[Mod(S_{0,4}) \cong PSL(2,\mathbb{Z}) \ltimes (\mathbb{Z}_2 \times \mathbb{Z}_2),\]
\end{prop}

\begin{rem}
The subgroup $\mathbb{Z}_2 \times \mathbb{Z}_2$ is generated by two elements of order 2, $c_1$ and $c_2$ shown as Figure \ref{c1c2}. The subgroup $PSL(2,\mathbb{Z})$ is generated by two half Dehn twists $\bar{\alpha}$ and $\bar{\beta}$ shown in Figure \ref{alphabeta}. $Mod(S_{0,4})=<\bar{\alpha}, \bar{\beta}, c_1, c_2>$.

\begin{figure}[H]
\centering
\subfigure[$c_1$ and $c_2$]{\label{c1c2}
\includegraphics[width=0.2\textwidth]{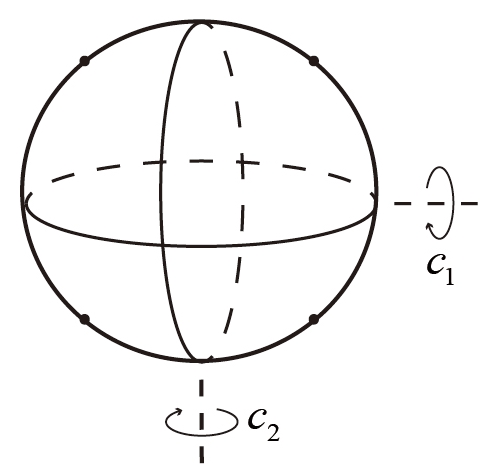}}\quad
\subfigure[Half Dehn twists $\bar{\alpha}$ and $\bar{\beta}$]{\label{alphabeta}
\includegraphics[width=0.72\textwidth]{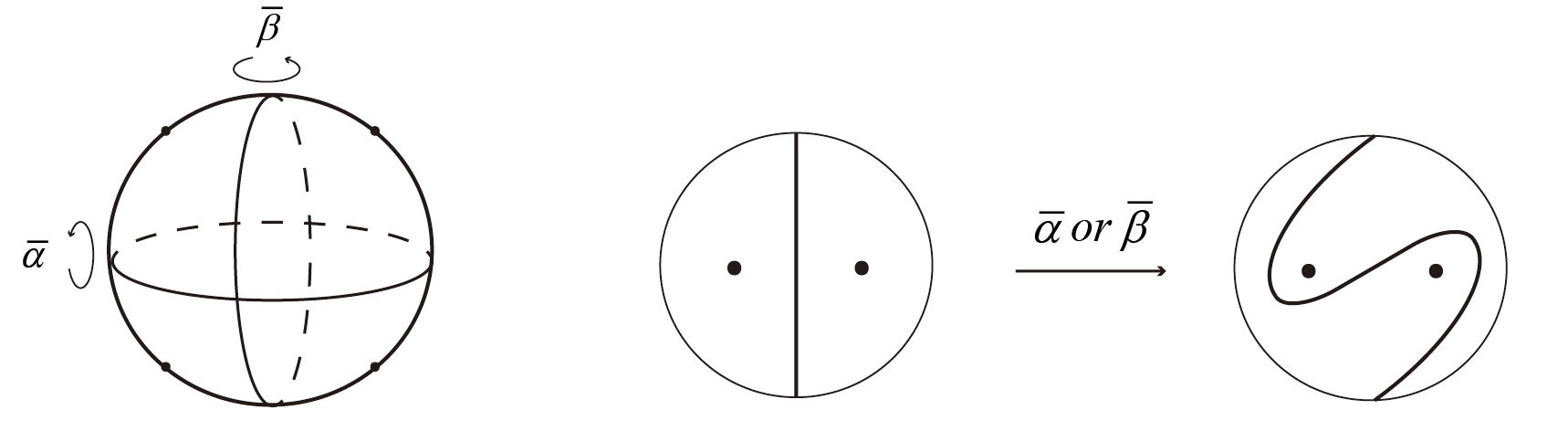}}
\caption{The generators of $Mod(S_{0,4})$}
\end{figure}

\end{rem}

\begin{defn}
Given a tangle $X=(M,t)$ and a Conway sphere $S$ in it, split the tangle $X$ along $S$ into two pieces $M_1$ and $M_2$. Let $h: S \rightarrow S^2$ be a homeomorphism such that $h(S \cap t)=P=4 \,\, points$. The tangle $X'=M_1 \cup_{h^{-1}gh} M_2$, where $g: (S^2, P) \rightarrow (S^2, P)$ such that $[g] \in  \langle c_1,c_2 \rangle - \{1\} \in Mod(S_{0,4})$, is called a {\it mutant} of $X$, and the operation of replacing $X$ by $X'$ is called {\it mutation} of $X$ along $S$.
\end{defn}

\begin{proof}[proof of Theorem \ref{mainthm}]

Assume that $b_1=b(a,b)$ for a pair of relatively prime integers $(a,b)$ satisfying $0<\frac{b}{a}\leq 1 (or \,\, \frac{b}{a}=\frac{1}{0}=\infty)$. As discussed above, the solution of this system of tangle equations exists only if there exist two pairs of relatively prime integers $(p_1,q_1)$ and $(p,q)$ satisfying that $p>1,p_1>1,|aq_1-bp_1|>1$ and $q=pp_1q_1\pm 1$ such that $\widetilde{b_2} \# \widetilde{b_3}=L(p,q) \# L(a \pm ap_1q_1p \mp bp_1^2p, b \pm aq_1^2p \mp bp_1q_1p)$. In this case, $O$ should be a tangle whose double branched cover is $\widetilde{O}$, and $\widetilde{O}=L(a,b) \setminus N(K)$ where $K$ is a $(p,q)$-cable of $K'=T_{p_1,q_1}$ lying in $T_1$ of $L(a,b)=T_1 \cup T_2$.

We first construct a tangle satisfying that its double branched cover is $\widetilde{O}$. Then we will show that any tangle whose double branched cover is $\widetilde{O}$ is homeomorphic to the tangle we construct. As discussed in Proposition \ref{Prop2}, there are two cases.

Case (\romannumeral1): $q=pp_1q_1+1$.

Let $f:S^1 \times D^2 \rightarrow N(K')$ be the homeomorphism defined as above. Then $T=f(S^1 \times \partial D^2)$ is an essential torus in $\widetilde{O}$ and also the only one essential torus in $\widetilde{O}$. Splitting $\widetilde{O}$ along the essential torus $T$, we obtain two manifolds $M=L(a,b) \setminus N(K')$ and $C_{p,q}$. By Corollary \ref{Lens space SFS2}, $M=M(0,1;(p_1,-e),(aq_1-bp_1,ad-be))$ where $p_1d-q_1e=1$. The cable space $C_{p,q}$ is the Seifert fiber space $M(0,2;(p,1))$. According to the results about double branched covers listed in Section \ref{section of Double branched covers}, $M$ is the double branched cover of a tangle $Q$ shown as Figure \ref{tangleQ}. We denote the associated standard involution of $M$ by $\sigma_1$. Meanwhile, $C_{p,q}=M(0,2;(p,1))$ is the double branched cover of a tangle $P$ shown in Figure \ref{tangleP}, which is a Montesinos pair in $S^2 \times I$. The associated standard involution of $C_{p,q}=M(0,2;(p,1))$ is denoted by $\sigma_2$. Restrict $\sigma_1$(resp.$\sigma_2$) on the torus boundary of $M$(resp.$C_{p,q}$), it is the so-called standard involution of a torus. In fact, given an isotopy class of homeomorphisms of the torus, there exists a representative which commutes with the standard involution, so we can extend the involutions $\sigma_1$ and $\sigma_2$ of $M$ and $C_{p,q}$ to an involution $\sigma$ of $\widetilde{O}$. In other words, there is a tangle $O_1$ which is obtained by gluing $Q$ and $P$ together satisfying that its double branched cover is $\widetilde{O}$. Let $f:\partial M \rightarrow \partial C_{p,q}$ be the gluing map to obtain $\widetilde{O}$, then the gluing map $\bar{f}: \partial Q \rightarrow \partial P$ to give $O_1$ is induced by $f$. Here we just give a tangle $O_1$ shown in Figure \ref{tangleO1}, and one can easily check $O_1$'s double branched cover is homeomorphic to $\widetilde{O}$ by carefully studying the two gluing maps $f$ and $\bar{f}$.

\begin{figure}
\centering
\subfigure[The tangle $Q$, where $A=\frac{-e}{p_1}$, $B=\frac{ad-be}{aq_1-bp_1}$, and $p_1d-q_1e=1$]
{\label{tangleQ}\includegraphics[width=0.29\textwidth]{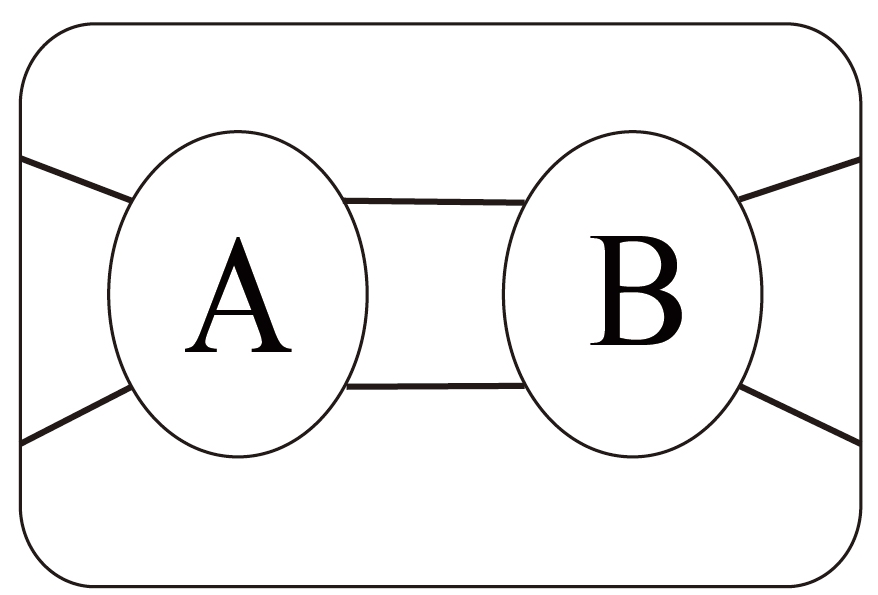}}\quad\qquad\qquad
\subfigure[The tangle $P$]
{\label{tangleP}\includegraphics[width=0.25\textwidth]{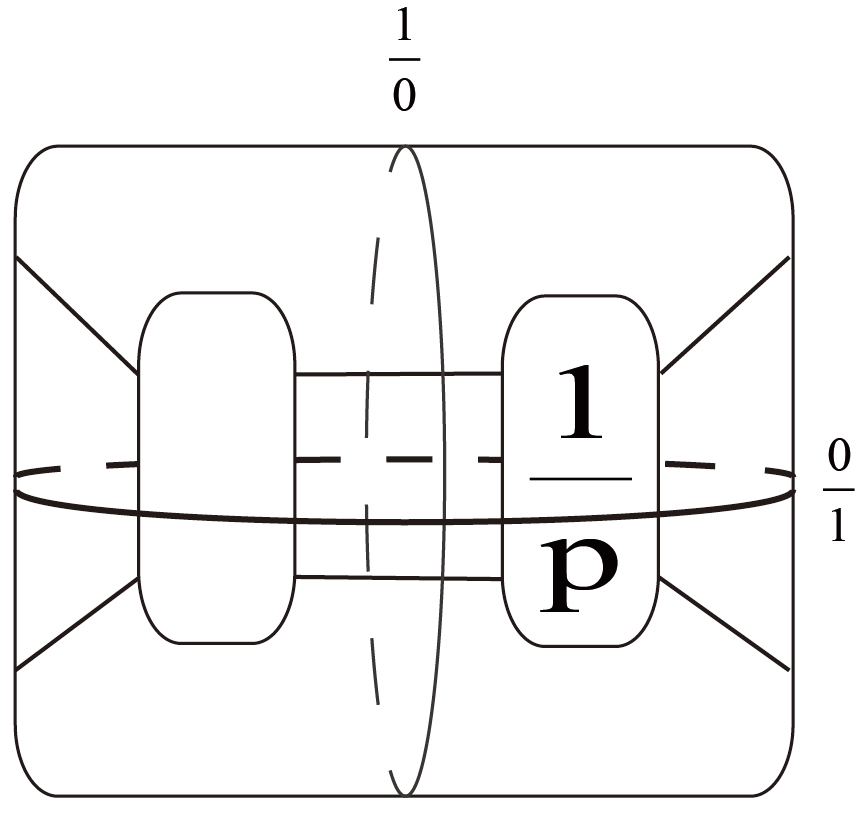}}
\caption{The tangle $Q$ and $P$}
\end{figure}

As discussed above, performing $\infty$-filling on $C_{p,q}$ (i.e.$\infty$-surgery along $K$) gives the original lens space $L(a,b)$, and this is equivalent to filling a $(1,0)$-fiber in $M(0,2;(p,1))$, namely $C_{p,q}(\infty)=M(0,1;(p,1),(1,0))$. Performing $pq$-filling on $C_{p,q}$ (i.e.$pq$-surgery along $K$) gives the connected sum of a solid torus and $L(p,q)$, which is equivalent to filling a $(0,1)$-fiber in $M(0,2;(p,1))$, namely $C_{p,q}(pq)=M(0,1;(p,1),(0,1))$. The two Dehn filling slopes on $\partial C_{p,q}$ are mapped respectively to two slants $\frac{0}{1}$ and $\frac{1}{0}$ on the corresponding boundary component of the tangle $P$ by the covering map induced by $\sigma_2$. In Figure \ref{tangleP}, the thick curve stands for the $\frac{0}{1}$ slant, and the thin curve is the $\frac{1}{0}$ slant. After gluing $P$ and $Q$ together, the two corresponding slants on $\partial O_1$ are shown in Figure \ref{tangleO1}. It means adding rational tangles which have the two slants as meridians respectively to $O_1$ give the 2-bridge link and the connected sum of two 2-bridge links we want. There is no other slant satisfying this since there is no other Dehn surgery slope along $K$ giving the original $L(a,b)$ or a non-prime manifold. In fact, only the Dehn filling slope $r=pq$ gives a non-prime manifold. To obtain the original $L(a,b)$, it probably happens only in the case that the slope $r=\frac{m}{n}$ and $m=npq \pm 1$, by Lemma \ref{Gordon2}, since the other two cases in Lemma \ref{Gordon2} produce a non-prime manifold or a toroidal manifold. Using the formula in Lemma \ref{Gordon1},
\[K(r)=L(a,b)\setminus N(T_{p_1,q_1})(m/np^2).\]
According to Lemma \ref{Lens space SFS3},
\[L(a,b) \setminus N(T_{p_1,q_1})(m/np^2)=M(0,0;(p_1,-e),(aq_1-bp_1,ad-de),(m-np^2p_1q_1,np^2))\]
where $p_1d-q_1e=1$. $M(0,0;(p_1,-e),(aq_1-bp_1,ad-de),(m-np^2p_1q_1,np^2))=L(a,b)$ if and only if $m-np^2p_1q_1=1$ and $np^2=0$ since $p_1>1$ and $|aq_1-bp_1|>1$. Thus $r=\frac{m}{n}=\frac{1}{0}=\infty$ since $p>1$. Therefore, as shown in Figure \ref{tangleO1}, only adding $0$-tangle and $\infty$-tangle give the links we want. One can easily check that $N(O_1+0)=b(a,b)$ and $N(O_1+\infty)=b(p,q) \#  b(a+ap_1q_1p-bp_1^2p, b+aq_1^2p-bp_1q_1p)$. This gives a pair of solutions $(X_1,X_2)$ when $O=O_1$ for the given tangle equations.

\begin{figure}
\centering
\subfigure{\includegraphics[width=1.6in]{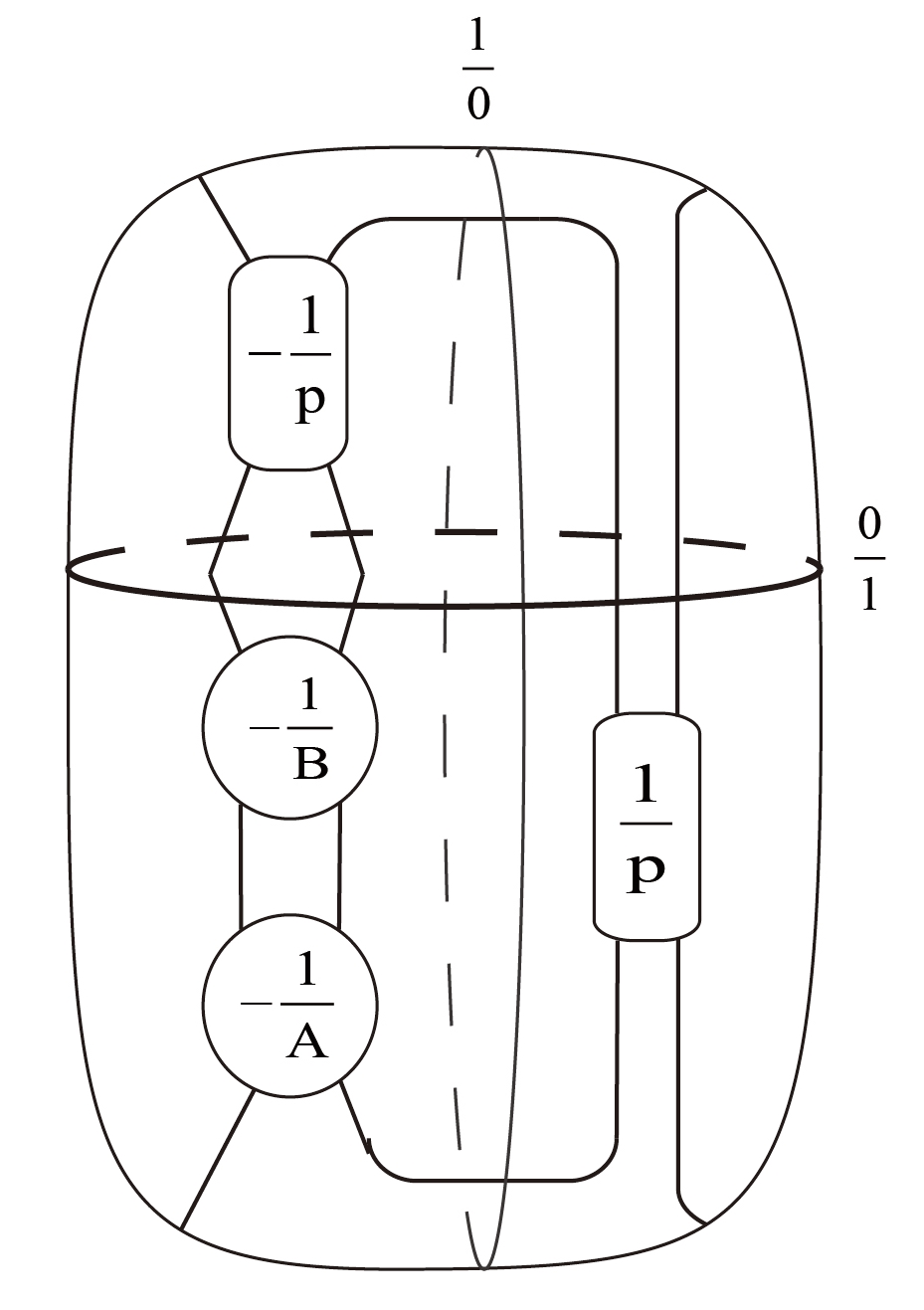}}
\subfigure{\includegraphics[width=0.18\textwidth]{solution2.jpg}}
\caption{The tangle $O_1$ and the corresponding pair of solutions $(X_1,X_2)$, where $A=\frac{-e}{p_1}$, $B=\frac{ad-be}{aq_1-bp_1}$, and $p_1d-q_1e=1$}\label{tangleO1}
\end{figure}

Now we denote $M$(resp.$C_{p,q}$) by $M_1$(resp.$M_2$). $\sigma_i$ is still the standard involution on $M_i$. $\sigma$ is the involution on $\widetilde{O}$ which is extended by $\sigma_1$ and $\sigma_2$. As discussed above, $\widetilde{O}/\sigma=O_1=M_1/\sigma_1 \cup_{\bar{f}} M_2/\sigma_2=Q\cup_{\bar{f}}P$.

Suppose $O'$ is a tangle whose double branched cover is homeomorphic to $\widetilde{O}$, and $\iota : \widetilde{O} \rightarrow \widetilde{O}$ is the associated involution, namely $\widetilde{O} / \iota=O'$. Since $T$ is the only one essential torus in $\widetilde{O}$, up to isotopy, we can assume that $T$ is invariant under $\iota$ by Theorem 8.6 in \cite{Scott}. $M_1$ is not homeomorphic to $M_2$, so $\iota$ preserves $T$, $M_1$ and $M_2$. Let $\iota_i$ be the restriction of $\iota$ to $M_i$.

$M_2$ is a Seifert fiber space over an annulus with one exceptional fiber. Restrict $\iota_2$ to $\partial \widetilde{O} \subset \partial M_2$, then it is the standard involution on this tours boundary since $\widetilde{O} / \iota$ is a tangle. Lemma 3.8 in \cite{Gordon1} tells us there exists a homeomorphism $\phi_2 : M_2 \rightarrow M_2$ isotopic to the identity such that $\iota_2=\phi^{-1}_2 \sigma_2 \phi_2$. Then $\iota$ restricted to $T$ is also the standard involution for $T$ as a torus. Therefore $T / (\iota\mid_{T})$ is a Conway sphere in $O'$. Then $M_1$ is the double branched cover of the tangle $O'- M_2 / \iota_2$, i.e. the tangle inside the Conway sphere $T / (\iota\mid_{T})$ in $O'$. According to Proposition 2.8 in \cite{Paoluzzi}, there is only one involution up to conjugation on $M_1$ satisfying that $M_1 /$(the involution) is a tangle, and $\sigma_1$ is such an involution. So there exists a homeomorphism $\phi_1 : M_1 \rightarrow M_1$ such that $\iota_1=\phi^{-1}_1 \sigma_1 \phi_1$. It will be shown that we can choose a $\phi_1$ such that $\partial \phi_1 : \partial M_1 \rightarrow \partial M_1$ is isotopic to the identity.

We can assume that $\phi_1$ is orientation-preserving, since we can easily find an orientation-reversing homeomorphism which commutes with $\sigma_1$. Also we assume $\phi_1$ preserves the orientation of fiber by multiplying with $\sigma_1$ or not. By Proposition 25.3 in \cite{Johannson}, the mapping class group of Seifert fiber space with orientable orbit surface except some special cases is the semidirect product of "vertical subgroup" and the extended mapping class group of the orbit surface. The "vertical subgroup" is generated by some Dehn twists along vertical annulus or torus and acts trivially on the orbit surface denoted by $F$. Besides, by the extended mapping class group of the orbit surface $F$ we mean the group of homeomorphisms of $F$ which send exceptional points to exceptional points with the same coefficients, modulo isotopies which are constant on the exceptional points. Another useful result is that the "vertical subgroup" is isomorphic to $H_1(F, \partial F)$. For $M_1$, the "vertical subgroup" is trivial since the first relative homology group of its orbit surface is trivial. Thus, the mapping class group of $M_1$ is isomorphic to the extended mapping class group of its orbit surface. We already have assumed $\phi_1$ preserves the orientations of $M_1$ and fiber, then $\partial \phi_1$ must be isotopic to the identity.

As our assumption, $\widetilde{O_1} / \iota = M_1/\iota_1 \cup_h M_2/\iota_2$ for some $h: \partial (M_1/\iota_1) \rightarrow \partial (M_2/\iota_2)$ satisfying $f: \partial M_1 \rightarrow \partial M_2$ is a lift of $h$. There exists $\phi_i: M_i \rightarrow M_i$ such that $\iota_i=\phi^{-1}_i \sigma_i \phi_i$ with $\partial \phi_i$ isotopic to the identity, for $i=1,2$. Then $\phi_i$ induces a homeomorphism $\bar{\phi_i} :(M_i/\iota_i) \rightarrow (M_i/\sigma_i)$. Restrict $\bar{\phi_i}$ to the boundary, we have the following commutative diagram:

\begin{displaymath}
\xymatrix{
\partial (M_1/\iota_1) \ar[d]^{\partial \bar{\phi_1}} \ar[rr]^{h} && \partial (M_2/\iota_2)  \ar[d]^{\partial \bar{\phi_2}}\\
\partial (M_2/\sigma_2) \ar[rr]^{\partial \bar{\phi_2} \circ h \circ \partial \bar{\phi_1}^{-1}} && \partial (M_2/\sigma_2).}
\end{displaymath}

This induces a homeomorphism:
\[\bar{\phi}=\bar{\phi_1} \cup \bar{\phi_2}: M_1/\iota_1 \cup_h M_2/\iota_2 \rightarrow M_1/\sigma_1 \cup_{\partial \bar{\phi_2} \circ h \circ \partial \bar{\phi_1}^{-1}} M_2/\sigma_2.\]
$\partial \bar{\phi_2} \circ h \circ \partial \bar{\phi_1}^{-1}$ lifts to $\partial \phi_2 \circ f \circ \partial \phi^{-1}_1$ which is isotopic to $f$ since $\partial \phi_i$ is isotopic to the identity. In fact, only the lift of $\bar{f}k$ is isotopic to $f$, where $k \in \langle c_1,c_2 \rangle \in Mod(\partial{Q})$ and $Mod(\partial{Q})$ represents the mapping class group of $\partial Q$ as a four-times-punctured sphere. Thus $\partial \bar{\phi_2} \circ h \circ \partial \bar{\phi_1}^{-1}=\bar{f}k$. Then,
\begin{align*}
O'&=\widetilde{O_1}/\iota=M_1/\iota_1 \cup_h M_2/\iota_2 \\
  &\cong M_1/\sigma_1 \cup_{\bar{f}k} M_2/\sigma_2=Q \cup_{\bar{f}k} P.
\end{align*}

In fact, gluing the tangle $Q$ and $P$ together by the gluing map $\bar{f}k:\partial Q \rightarrow \partial P$ is equivalent to performing mutations on $O_1$ along the dotted Conway sphere $S$ shown as Figure \ref{tangleOb1b}. The tangle in the Conway sphere $S$ in $O_1$ is a Montesinos tangle which is invariant under some rotations in $\langle c_1,c_2 \rangle$. Thus we only obtain two tangles by mutations, $O_1$ itself and $O_2$ shown in Figure \ref{thetangleO2}. We can easily show that $O_2$ is homeomorphic to $O_1$, by extending the operation $c_1 \in Mod(\partial O_1)$ to the whole $O_1$. Thus $O' \cong O_1$.

\begin{figure}[H]
\centering
\subfigure[The tangle $O_1$]
{\label{tangleOb1b}\includegraphics[width=0.25\textwidth]{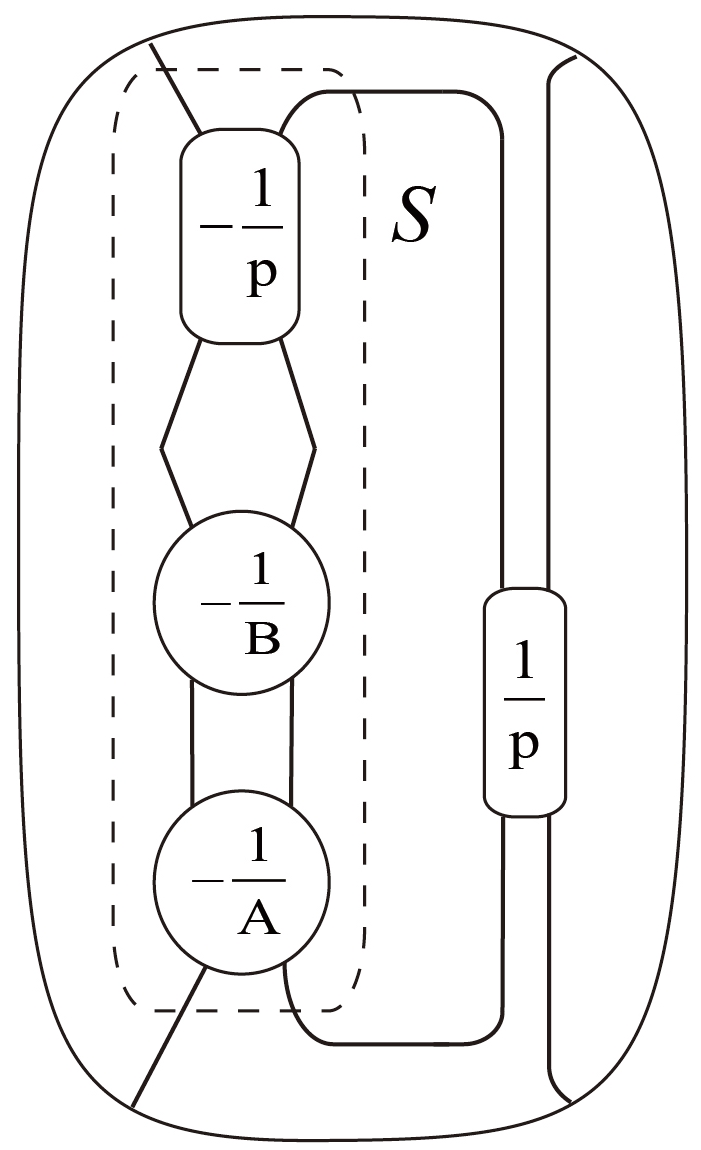}}\quad\quad\quad
\subfigure[The tangle $O_2$]
{\label{thetangleO2}\includegraphics[width=0.248\textwidth]{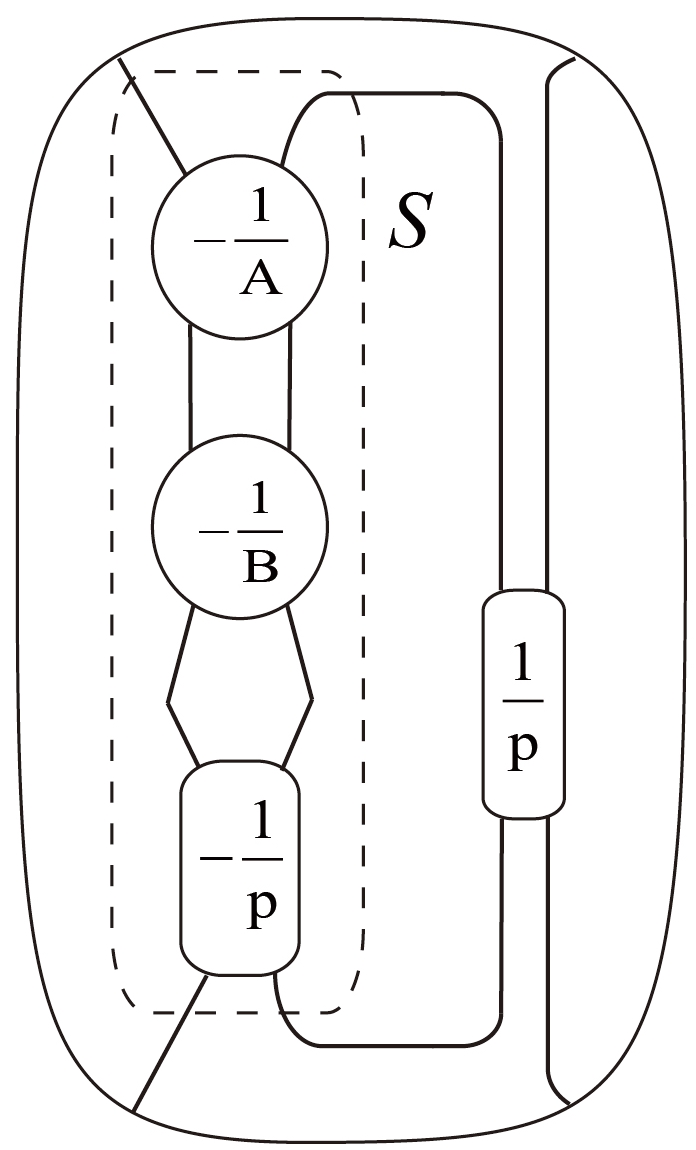}}
\subfigure
{\includegraphics[width=0.06\textwidth]{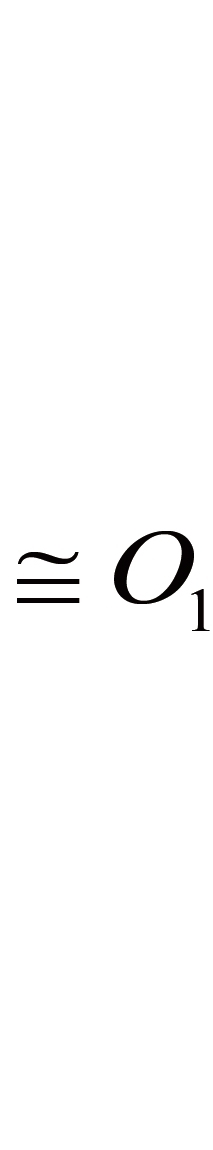}}
\caption{The tangle $O_1$ and $O_2$}\label{O1}
\end{figure}

All the homeomorphisms on $O_1$ can be induced by homeomorphisms on boundary of $O_1$. These homeomorphisms on $\partial O_1$ also induce a new pair of slants corresponding to the new pair of solutions $(X_1,X_2)$ for the new tangle obtained. This pair of solutions $(X_1,X_2)$ is unique, for otherwise there exist other Dehn fillings giving the manifolds we want. Actually, we only choose orientation-preserving homeomorphisms. Besides, we only want the pair of solutions with $X_1=0$ as mentioned at the beginning of this section, thus the homeomorphisms we can perform on $\partial O_1$ are limited and can be easily worked out. In fact, only $\bar{\beta}^n k \, \in Mod(\partial O_1)$ could preserve the meridian of $X_1=0$-tangle, where $k\, \in \langle c_1,c_2 \rangle$, and $n \in \mathbb{Z}$. Therefore, all the solutions up to equivalence are
\begin{align*}
O&=k(O_1) \circ (n,0)\\
X_1&=0-tangle, \\
X_2&=\infty \circ (-n,0)-tangle,
\end{align*}
where $k(\ast)$ means extending the operation $k$ on the boundary of the tangle $\ast$ to the whole $\ast$. In \cite{Darcy}, it is shown that the solutions $(O \circ (n,0), X_1=0, X_2=\infty \circ (-n,0))$ is equivalent to $(O, X_1=0, X_2=\infty)$. Thus $(O=k(O_1), X_1=0, X_2=\infty)$ give all the solutions, up to equivalence, as shown in the case (\romannumeral1) of this theorem. Obviously, $N(k(O_1)+0)=b(a,b)$ and $N(k(O_1)+\infty)=b(p,q) \#  b(a+ap_1q_1p-bp_1^2p, b+aq_1^2p-bp_1q_1p)$.

Case (\romannumeral2): $q=pp_1q_1-1$.

This case is similar to the case (\romannumeral1). Split $\widetilde{O}$ along the only one essential torus $T$, we obtain two pieces $M=L(a,b)\setminus N(K')$ and $C_{p,q}$. $M$ is still the Seifert fiber space $M(0,1;(p_1,e),(aq_1-bp_1,ad-be))$, which is the double branched cover of the tangle $Q$ in Figure \ref{tangleQ}. $C_{p,q}=M(0,2;(p,-1))$ is the double branched cover of the tangle $P'$ shown in Figure \ref{tanglePpri}. Then we construct a tangle $O'_1$ whose double branched cover is $\widetilde{O}$, by gluing $P'$ and $Q$ together. The pair of solutions $(X_1,X_2)$ when $O=O'_1$ is given in Figure \ref{tangleO'1X1X2}, similarly by studying the surgeries on $\widetilde{O}$ and $C_{p,q}$. Besides, when $O=O'_1$ the pair of solutions $(X_1,X_2)$ is unique. The same method as in case (\romannumeral1) \, can be used to prove any tangle whose double branched cover is $\widetilde{O}$ is homeomorphic to $O'_1$. Similarly, $(O=k(O'_1),X_1=0,X_2=\infty)$ give all the solutions up to equivalence as shown in the case (\romannumeral2) of this theorem. One can check that $N(k(O'_1)+0)=b(a,b)$ and $N(k(O'_1)+\infty)=b(p,q) \#  b(a-ap_1q_1p+bp_1^2p, b-aq_1^2p+bp_1q_1p)$.

\begin{figure}[H]
\centering
\subfigure[The tangle $P'$]
{\label{tanglePpri}\includegraphics[width=0.26\textwidth]{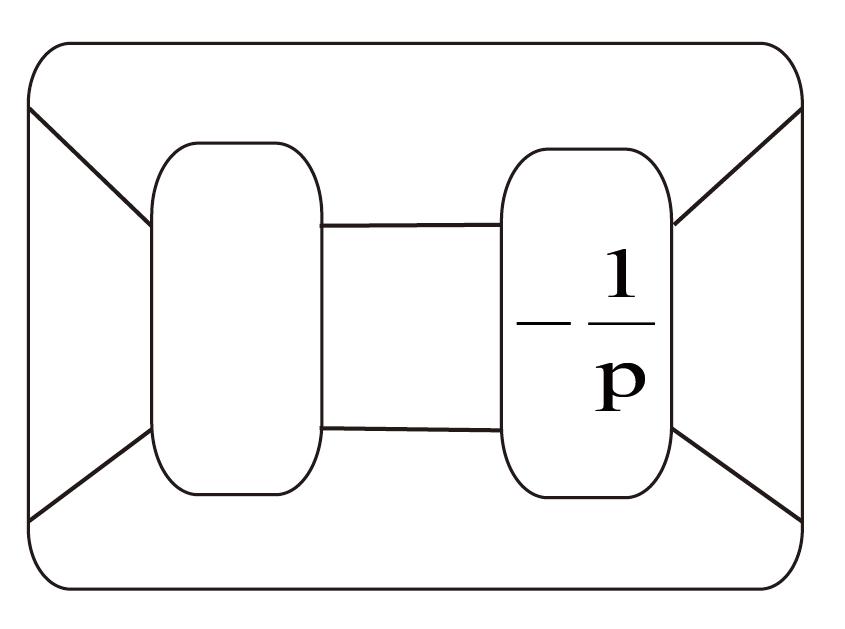}}\quad\quad
\subfigure[The tangle $O'_1$]
{\label{tangleOpri1}\includegraphics[width=0.25\textwidth]{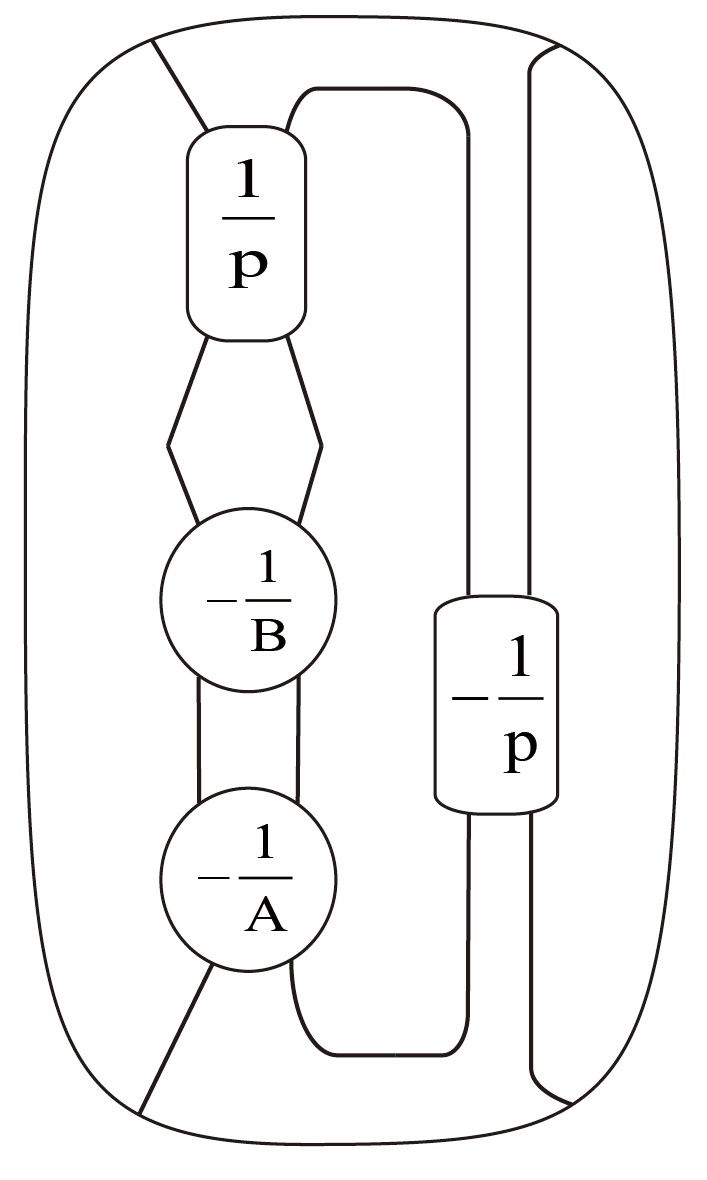}}\quad\quad
\subfigure
{\includegraphics[width=0.18\textwidth]{solution2.jpg}}
\caption{The tangle $O'_1$ and the corresponding pair of solutions $(X_1,X_2)$, where $A=\frac{-e}{p_1}$, $B=\frac{ad-be}{aq_1-bp_1}$, and $p_1d-q_1e=1$.}\label{tangleO'1X1X2}
\end{figure}

\end{proof}

\begin{rem}
In \cite{Gordon1}, Gordon proved that any tangle whose double branched cover is homeomorphic to that of EM-tangle is homeomorphic to the EM-tangle by a similar method, where EM-tangle has a similar structure as our tangle $O_1$. Besides, Paoluzzi's method in \cite{Paoluzzi} can be used to prove that there are at most 4 tangles whose double branched cover are $\widetilde{O}$ and the 4 tangles are obtained by mutations of the tangle $O_1$.
\end{rem}

\subsubsection{$\widetilde{O}$ is a Seifert fiber space}

Buck and Mauricio's paper \cite{Buck} also includes this case, while it assumes that neither of $b_2$ and $b_3$ is $b(0,1)$ (i.e.the unlink). Besides, our definition of tangle is different from \cite{Buck}, since we allow tangle to have circles embedded in. Here we will give a theorem without such an assumption.

\begin{thm1}
\label{mainthm2}
Suppose
\begin{align*}
N(O+X_1)&=b_1          \\
N(O+X_2)&=b_2 \# b_3,
\end{align*}
where $X_1$ and $X_2$ are rational tangles, and $b_i$ is a 2-bridge link, for $i=1,2,3$, with $b_2$ and $b_3$ nontrivial. Suppose $\widetilde{O}$ is a Seifert fiber space, then the system of tangle equations has solutions if and only if one of the following holds:

(\romannumeral1) There exist 2 pairs of relatively prime integers $(a,b)$ and $(p,q)$ satisfying $0<\frac{b}{a}\leq 1 (or \,\, \frac{b}{a}=\frac{1}{0}=\infty), p>1$, and $|aq-bp|>1$ such that $b_1=b(a,b)$ and $b_2 \# b_3=b(p,-e) \# b(aq-bp, ad-be)$, where $pd-qe=1$ (Note that choosing different $d$ and $e$ such that $pd-qe=1$ has no effect on results). Solutions up to equivalence are shown as the following:

\begin{figure}[H]
\centering
\subfigure[]{\label{SFSsolution1}
\includegraphics[width=0.4\textwidth]{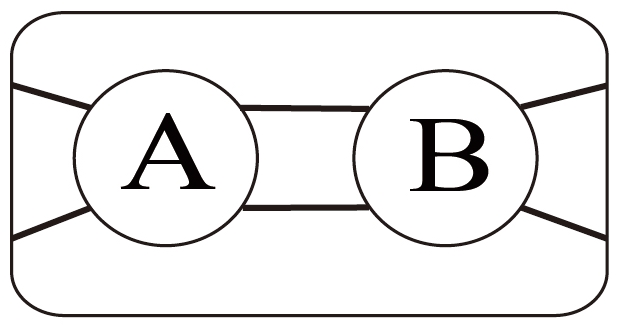}}\qquad\qquad
\subfigure{\label{SFSsolutionX}
\includegraphics[width=0.18\textwidth]{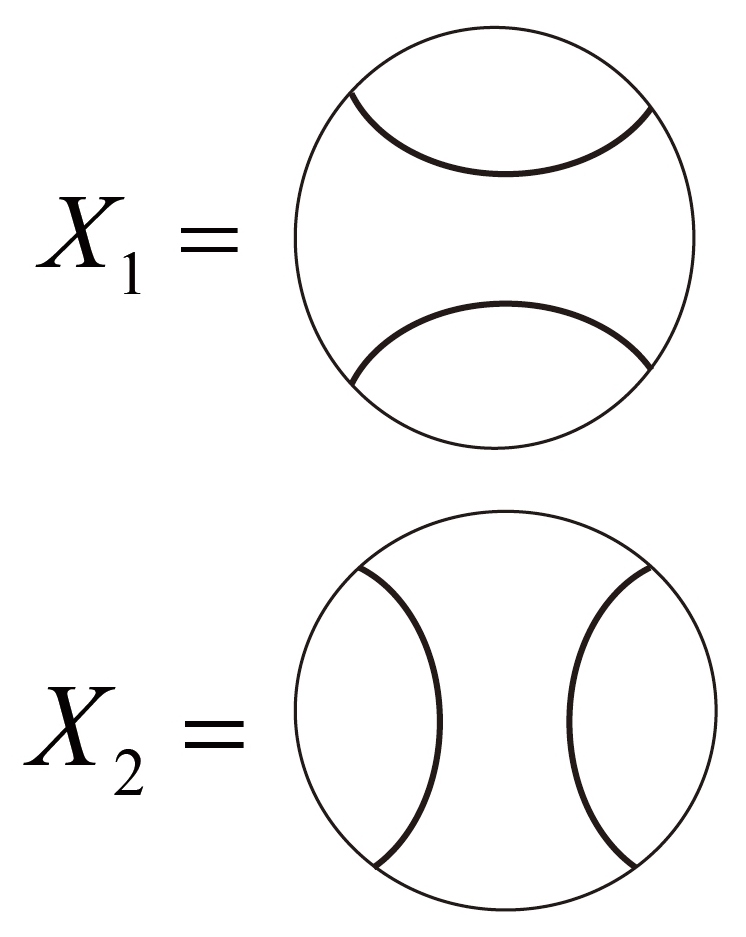}}
\caption{$O$ = the tangle in (a), where $A=\frac{-e}{p}$, $B=\frac{ad-be}{aq-bp}$ (or $A=\frac{ad-be}{aq-bp}$, $B=\frac{-e}{p}$). $X_1=0$-tangle and $X_2=\infty$-tangle.}\label{SFSsolution12}
\end{figure}

(\romannumeral2)There exists an integer $p$ satisfying $|p|>1$ such that $b_1=b(4p, 1-2p)$ and $b_2 \# b_3=b(0,1) \# b(p,1)$. Solutions up to equivalence are shown as the following:

\begin{figure}[H]
\centering
\subfigure[]{\label{SFSsolution2}
\includegraphics[width=0.4\textwidth]{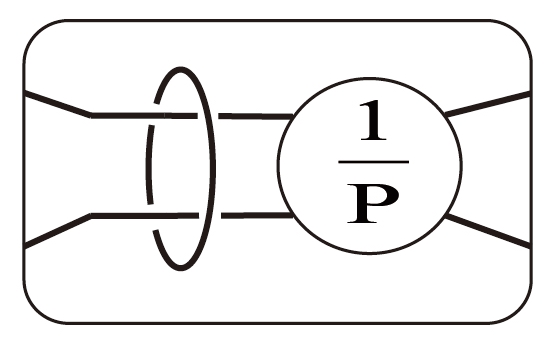}}\qquad\qquad
\subfigure{\label{SFSsolutionX}
\includegraphics[width=0.18\textwidth]{SFSsolutionX.jpg}}
\caption{$O=R + \frac{1}{p}$ the tangle shown in (a), where $R$ is the ring tangle. $X_1=0$-tangle and $X_2=\infty$-tangle.}\label{SFSsolution21}
\end{figure}

\end{thm1}

\begin{proof}
Suppose that $b_1=b(a,b)$ for a pair of relatively prime integers $(a,b)$ satisfying $0<\frac{b}{a}\leq 1 (or \,\, \frac{b}{a}=\frac{1}{0}=\infty)$, then $\widetilde{b_1}=L(a,b)$. $\widetilde{X_1}$ is a solid torus lying in $L(a,b)$, and we regard it as a tubular neighborhood of a knot $K$ in $L(a,b)$. Since we assume that $\widetilde{O}=L(a,b)\setminus N(K)$ is a Seifert fiber space, $K$ is isotopic to a fiber in some generalized Seifert fibration of $L(a,b)$. According to Lemma \ref{Lens space SFS1}, there are two types of fibration for a lens space.

(\romannumeral1)$L(a,b)$ is fibered over $S^2$.

In fact, $K$ is isotopic to an ordinary fiber of this type of fibration, since otherwise $L(a,b)\setminus N(K)=\widetilde{O}$ is a solid torus, but there dose not exist rational tangle solution for $O$. Let $L(a,b)=T_1 \cup T_2$ where $T_i$ is a solid torus, for $i=1,2$. $K$ can be regarded as a $(p,q)$-torus knot in one of the solid tori of $L(a,b)$, without loss of generality $T_1$. In fact, $p\neq 0$ because otherwise either $\widetilde{O}=L(a,b)\setminus N(K)$ is reducible which contradicts our assumption, or $L(a,b)=S^3$. $L(a,b)=S^3$ is also impossible since if it is and $p=0$, then $L(a,b)\setminus N(K)=$solid torus. Besides, $p \neq 1$ for otherwise $L(a,b)\setminus N(K)=$solid torus. Thus $p>1$.  By Corollary \ref{Lens space SFS2},
\[L(a,b)=M(0,0;(p,-e), (aq-bp,ad-be)),\]
where $pd-qe=1$. $|aq-bp|>1$ for the same reason as $p>1$ since a $(p,q)$-torus knot in $T_1$ of $L(a,b)$ is also a torus knot with winding number $|aq-bp|$ in $T_2$ of $L(a,b)$. Since $K$ is isotopy to an ordinary fiber,
\[\widetilde{O}=L(a,b)\setminus N(K)=M(0,1;(p,-e), (aq-bp,ad-be)).\]
We choose a longitude-meridian basis for $K$ using the same principle as choosing basis for $J$ in $S^1 \times D^2$. According to Lemma \ref{Lens space SFS3}, only doing $\infty$-surgery (resp.$pq$-surgery) along $K$ produces $L(a,b)$ (resp.\,a non-prime manifold), and $\infty$-surgery (resp.$pq$-surgery) is equivalent to filling a $(1,0)$-fiber (resp.$(0,1)$-fiber) in $M(0,1;(p,-e), (aq-bp,ad-be))$. That is,
\[K(\infty)=M(0,0;(p,-e), (aq-bp,ad-be), (1,0))=L(a,b)\]
\[(resp. K(pq)=M(0,0;(p,-e), (aq-bp,ad-be), (0,1))=L(p,-e) \# L(aq-bp,ad-be)).\]

According to the results about double branched covers listed in Section \ref{section of Double branched covers}, the tangle $O_1$ shown in Figure \ref{tangleO14SFS} satisfies that its double branched cover is $\widetilde{O}$. Let $v$ denote the associated standard involution on $\widetilde{O}$, then $\widetilde{O}/v=O_1$. As our analysis about surgeries above, filling along two slopes $\frac{0}{1}$ and $\frac{1}{0}$ (i.e.filling $(1,0)$-fiber and $(0,1)$-fiber in $M(0,1;(p,-e), (aq-bp,ad-be))$) produce the manifolds we want. The two Dehn filling slopes are mapped to $\frac{0}{1}$ and $\frac{1}{0}$ slants respectively on $\partial O_1$ by the covering map induced by $v$. It means adding the tangle $X_1=0$-tangle and $X_2=\infty$-tangle give the 2-bridge link and the connected sum of two 2-bridge links we want. One can check that $N(O_1+0)=b(a,b)$ and $N(O_1+\infty)=b(p,-e) \# b(aq-bp, ad-be)$. There is no other slant satisfying this, since there is no other surgery slope along $K$ giving the original $L(a,b)$ or a non-prime manifold.

\begin{figure}[H]
\centering
\subfigure[The tangle $O_1$]{\label{tangleO14SFS}
\includegraphics[width=0.35\textwidth]{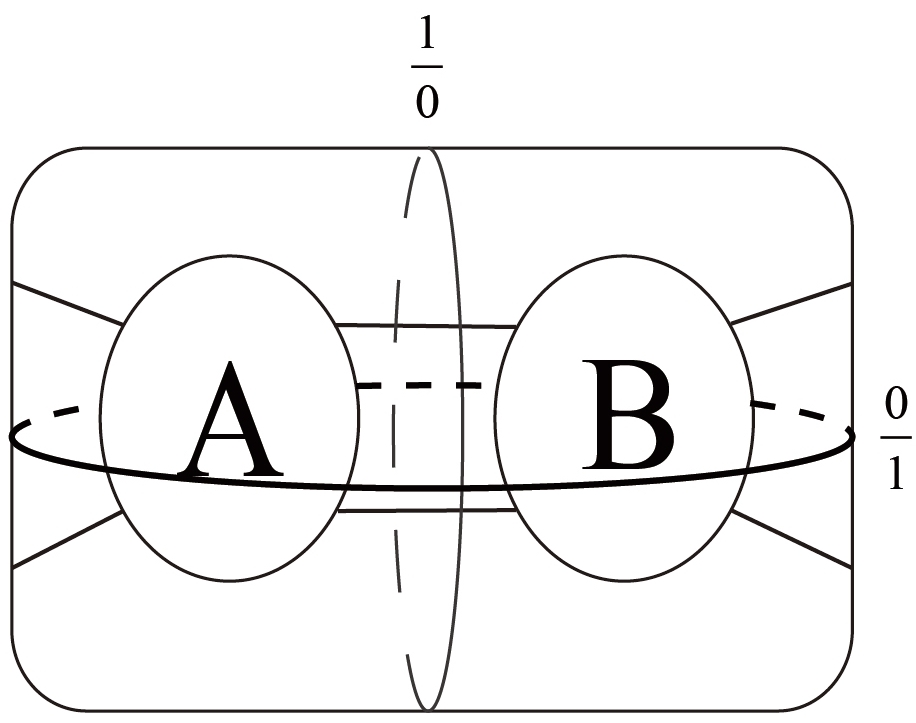}}\qquad\qquad
\subfigure{
\includegraphics[width=0.16\textwidth]{SFSsolutionX.jpg}}
\caption{The tangle $O_1=A+B$ and the corresponding pair of solutions $(X_1,X_2)$,  where $A=\frac{-e}{p}$, $B=\frac{ad-be}{aq-bp}$ and $pd-qe=1$,}\label{SFSsolution12}
\end{figure}

By Proposition 2.8 in \cite{Paoluzzi}, there is only one involution on $\widetilde{O}$ (i.e.$v$), up to conjugation, satisfying $\widetilde{O} / (the \, involution)$ is a tangle. Therefore, any tangle whose double branched cover is $\widetilde{O}$ is homeomorphic to the tangle $O_1$. Since we expect $X_1=0$-tangle,
\[O=k(O_1), \, X_1=0, \, X_2=\infty,\]
where $k \in \langle c_1,c_2 \rangle \in Mod(\partial O_1)$, give all the solutions up to equivalence, like analysis in Theorem \ref{mainthm}.

Obviously, $N(k(O_1)+0)=b(a,b)$ and $N(k(O_1)+\infty)=b(p,-e) \# b(aq-bp, ad-be)$. We know that performing $pq$-surgery gives a $L(p,q)$ summand. Actually, $b(p,-e)=b(p,q)$ since $pd-qe=1$.

(\romannumeral2)$L(a,b)$ is fibered over $\mathbb{R}P^2$.

Let $L(a,b)=T_1 \cup_g (S^1 \widetilde{\times}  M\ddot{o}bius \,\, band)$ where $T_1$ is a solid torus. If $K$ is isotopic to an exceptional fiber, then $L(a,b)\setminus N(K)$ has no exceptional fiber, so we can choose another fibration of $L(a,b)$ over $S^2$ such that $K$ is isotopic to an ordinary fiber of the new fibration, which has been discussed in the previous case. Therefore we assume that $K$ is isotopic to an ordinary fiber. We can regard $K$ as a $(p,q)$-torus knot lying in $T_1$ with $|p|>1$. Because if $|p|=0$, then $K$ is a knot in a ball, thus $L(a,b)\setminus N(K)$ is reducible; if $|p|=1$, then we can also refiber $L(a,b)$ over $S^2$ such that $K$ is isotopic to an ordinary fiber. Here we choose a longitude-meridian basis for $K$ by using the same principle as choosing basis for $J$ in $S^1 \times D^2$. By Corollary \ref{Lens space SFS2} and Lemma \ref{Lens space SFS1},
\[L(a,b)=M(-1,0;(p,1))=L(4p,1-2p) \,\, and \,\, q \cong 1 \,\, mod \, p.\]
In fact, $M(-1,0;(p,1)) \cong M(-1,0;(p,-1))$ with different orientation. Since $K$ is isotopic to an ordinary fiber, then
\[\widetilde{O}=L(a,b)\setminus N(K)= M(-1,1;(p,1)).\]
Using the formula in Lemma \ref{Lens space SFS4}, only doing $\infty$-surgery (resp.$pq$-surgery) along $K$ gives $L(a,b)=L(4p,1-2p)$ (resp.a non-prime manifold), and $\infty$-surgery (resp.$pq$-surgery) is equivalent to filling a $(1,0)$-fiber (resp.$(0,1)$-fiber) in this given fibration of $L(a,b)$. That is,
\[K(\infty)=M(-1,0;(p,1), (1,0))=L(4p,1-2p)\]
\[(resp. K(pq)=M(-1,0;(p,1),(0,1))=S^1 \times S^2 \#L(p,1)).\]

As shown in Section \ref{section of Double branched covers}, the double branched cover of the tangle $O_2$ shown in Figure \ref{tangleO24RP} is $\widetilde{O}$. Fillings along the two slopes $\frac{0}{1}$ and $\frac{1}{0}$ on the boundary of $M(-1,1;(p,1))$ (i.e.filling $(1,0)$- fiber and $(0,1)$-fiber in $M(-1,1,(p,1))$ respectively) give the manifolds we want as discussion above. The two Dehn filling slopes are mapped to two slants $\frac{0}{1}$ and $\frac{1}{0}$ respectively on $\partial O_2$ by the covering map, which give the pair of solutions $(X_1=0, X_2=\infty)$ when $O=O_2$. Besides, when $O=O_2$ the pair of solutions is unique by the analysis about surgeries above. One can easily check that $N(O_2+0)=b(4p,1-2p)$ and $N(O_2+\infty)=b(0,1) \# b(p,1)$.

\begin{figure}[H]
\centering
\subfigure[the tangle $O_2$]{\label{tangleO24RP}
\includegraphics[width=0.38\textwidth]{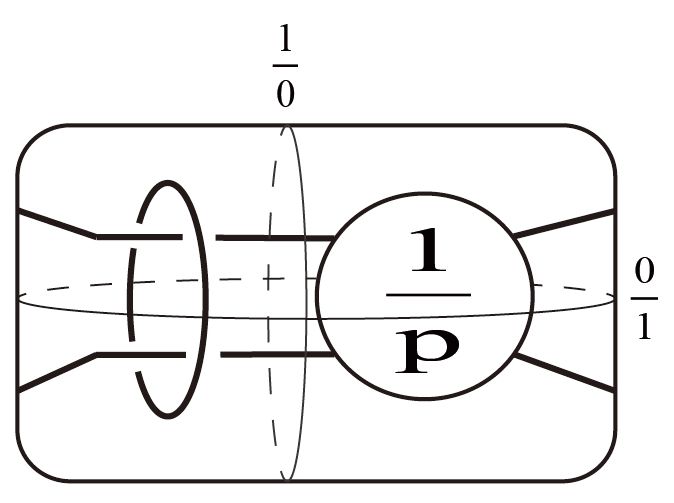}}\qquad\qquad
\subfigure{
\includegraphics[width=0.16\textwidth]{SFSsolutionX.jpg}}
\caption{The tangle $O_2$ and the corresponding pair of solutions $(X_1,X_2)$.}
\end{figure}

\begin{figure}[H]
\centering
\subfigure[The tangle $O_2$]{\label{tangleO2Sphere}
\includegraphics[width=0.28\textwidth]{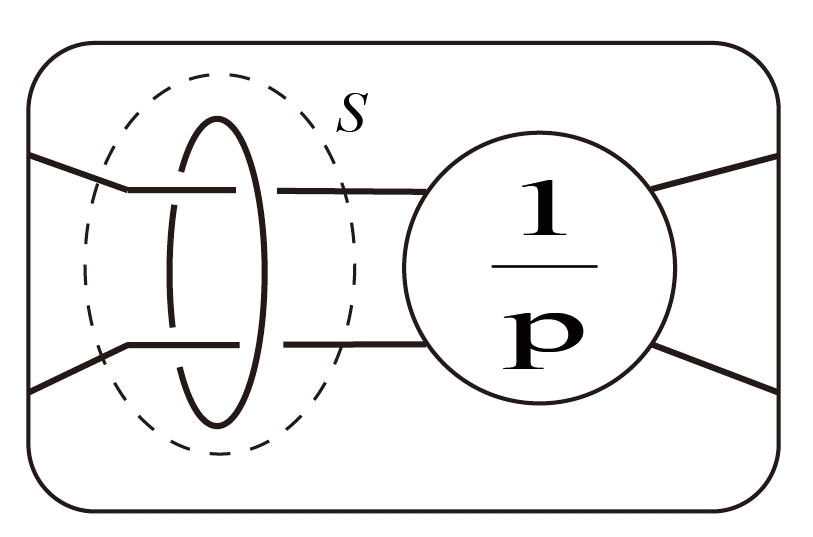}}\qquad
\subfigure[The tangle $U$]{\label{tangleU}
\includegraphics[width=0.20\textwidth]{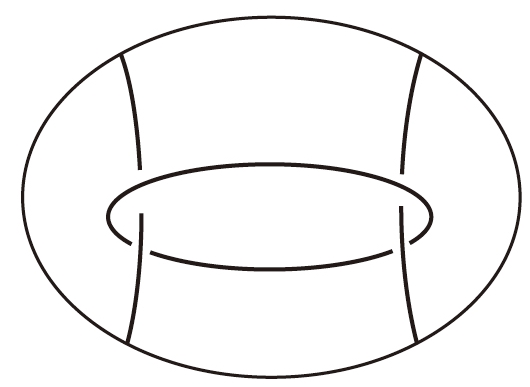}}\qquad
\subfigure[The Montesinos pair $M=(0,2; \varnothing, \frac{1}{p})$]{\label{tanglePN}
\includegraphics[width=0.28\textwidth]{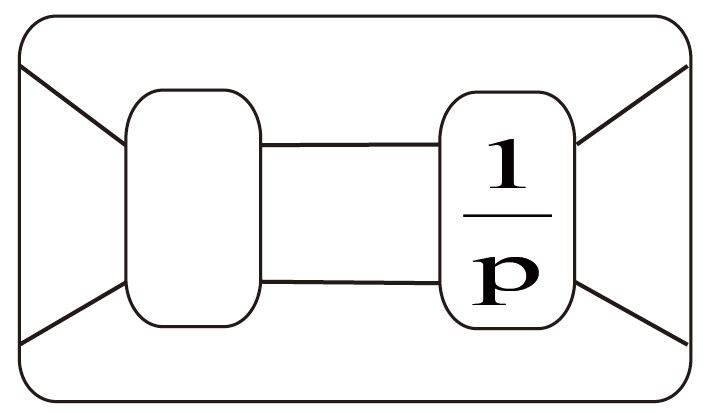}}
\caption{Splitting the tangle $O_2$ as $U \cup M$}\label{O2UM}
\end{figure}

Now we show that any tangle whose double branched cover is $\widetilde{O}$ is homeomorphic to $O_2$. Split the tangle $O_2$ into two tangles. One is the tangle, denoted by $U$, in the dotted Conway sphere $S$ in $O_2$ shown in Figure \ref{O2UM}, which is actually the ring tangle. Another one is the tangle outside $U$ in $O_2$, denoted by $M$, shown in Figure \ref{tanglePN}, which is a Montesinos pair in $S^2 \times I$. According to Section \ref{section of Double branched covers}, $\widetilde{U}=M(-1,1;) \cong M(0,1; (2,1),(2,-1))$, denoted by $M_1$, which can be regarded as a Seifert fiber space over a disk with two exceptional fibers, and $\widetilde{M}=M(0,2;(p,1))$, denoted by $M_2$, which is a Seifert fiber space over an annulus with one exceptional fiber. Then $\widetilde{O_2}=\widetilde{O}=M_1 \cup M_2$. The lift of the Conway sphere $S$, i.e.$M_1 \cap M_2$, is the only one essential torus in $\widetilde{O}$ up to isotopy. This is very similar to the situation in Theorem \ref{mainthm}. Therefore we can use the same method to show that any tangle whose double branched cover is $\widetilde{O}$ is homeomorphic to $O_2$ or a mutant of $O_2$ along the Conway sphere $S$. The tangle $U$ is so special such that $U=k(U)$ for any $k \in \langle c_1,c_2 \rangle \in Mod(\partial U)$, namely any mutant of $O_2$ is homeomorphic to $O_2$. Then any tangle whose double branched cover is $\widetilde{O}$ is homeomorphic to the tangle $O_2$. In addition, $O_2$ is invariant under any homeomorphism extended by $k \in \langle c_1,c_2 \rangle \in Mod(\partial O_2)$, then
\[O=O_2, \, X_1=0, \, X_2=\infty\]
give all the solutions up to equivalence.

\end{proof}

\section{Some other tangle equations}
\label{Some other tangle equations}

We can also solve the following system of tangle equations:
\begin{align}
N(U+X_1)&=b_1          \label{eq5}\\
N(U+X_2)&=b_2,         \label{eq6}
\end{align}
where $X_1$ and $X_2$ are rational tangles, $U$ is an algebraic tangle but not a generalized Montesinos tangle, $b_1$ and $b_2$ are 2-bridge links with $b_1 \neq b_2$.

In fact, the system of tangle equations has been discussed in many papers, like \cite{DarcySumners} and \cite{Darcy}. But they solve the equations under the assumption that $U$ is a generalized Montesinos tangle or $d(X_1,X_2)>1$ (if $d(X_1,X_2)>1$, then we have that $\widetilde{U}$ is a Seifert fiber space by Cyclic Surgery Theorem \cite{Culler}. It can be shown that $\widetilde{U}$ is a Seifert fiber space over a disk, thus $U$ is a generalized Montesinos tangle).

Also lifting to the double branched covers, the system of tangle equations is translated to
\begin{align}
&\widetilde{U}(\alpha)= the \,\, lens \,\, space \quad \widetilde{b_1}                                                        \label{eq7}\\
&\widetilde{U}(\beta)= the \,\, lens \,\, space \quad \widetilde{b_2},                                                          \label{eq8}
\end{align}

where $\widetilde{U}$ (resp.$\widetilde{b_i}$) denotes the double branched cover of $U$(resp.$b_i$) and $\alpha$ (resp.$\beta$) is the induced Dehn filling slope by adding rational tangle $X_1$(resp.$X_2$). Therefore, the problem turns out to be finding knots in the lens space $\widetilde{b_1}$ which admits a surgery to another lens space.

According to Section \ref{section of Double branched covers}, the double branched cover of an algebraic tangle is a graph manifold. Obviously the algebraic tangle $U$ is locally unknotted, since $b_1 \neq b_2$, both of which are prime. Also it is impossible for $U$ to contain a splittable unknot. Therefore, $\widetilde{U}$ is an irreducible graph manifold, but not a Seifert fiber space since $U$ is not a generalized Montesinos tangle.

Now we split $\widetilde{U}$ along its incompressible tori to study $\widetilde{U}$ and this is the idea from Buck and Mauricio \cite{Buck}. In fact, that all the tori in $\widetilde{U}$ are separating. Because if not, a non-separating torus is still non-separating after Dehn filling, which contradicts the fact that there is no non-separating torus in a lens space. Let $T$ be a collection of disjoint non-parallel incompressible tori such that each component of $\widetilde{U}|T$ is atoroidal. Here $\widetilde{U}$ is an irreducible graph manifold. After cutting it along $T$, we only have atoroidal Seifert fiber spaces (i.e.small Seifert fiber spaces) left.

\begin{defn}
A {\it splitting graph of $\widetilde{U}$ along $T$} is a graph $G$ which uses edges to represent the incompressible tori in $T$ and use vertices to represent the connected components of this decomposition. An edge connects two vertices if and only if the incompressible torus corresponding to the edge separates the two components corresponding to these two vertices.
\end{defn}

In fact, the splitting graph of $\widetilde{U}$ along $T$ is a tree, since all the tori in $T$ are separating. Choose the vertex whose corresponding component contains $\partial \widetilde{U}$ to be the root of this graph, and denote it by $v_0$. We define {\it the level of a vertex} to be the minimum number of edges of a path which connects this vertex and the root, and {\it the level of an edge} is defined to be the same as the level of the adjacent vertex which is closer to the root. Then we have the splitting graph of $\widetilde{U}$ is like the following:

\begin{figure}[H]
\centering
\includegraphics[width=0.65\textwidth]{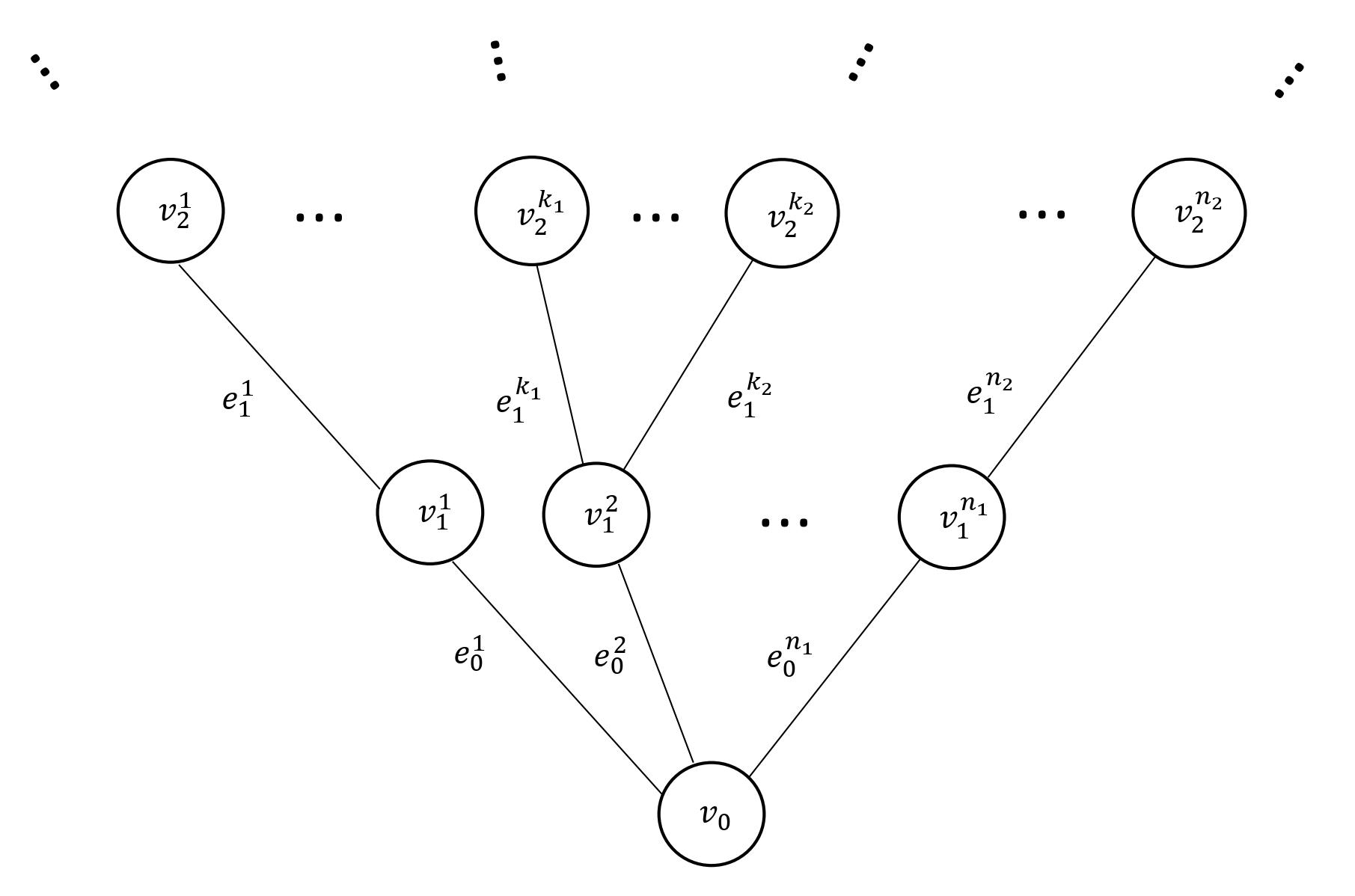}
\caption{\label{tree1} The splitting graph of $\widetilde{U}$ along $T$}
\end{figure}

We can verify the following lemma by almost the same argument as Proposition 5.8 in \cite{Buck} and the formula of Lemma \ref{Gordon2}.

\begin{lem}
\label{v_0}
$v_0$ is a Seifert fiber space over an annulus with exact one exceptional fiber. $v_0(\alpha)$ is a solid torus.
\end{lem}

\begin{prop}
\label{linear tree}
The splitting graph of $\widetilde{U}$ along $T$ is a linear tree as Figure \ref{lineartree} shown. Each component is a Seifert fiber space over an annulus with exact one exceptional fiber, except the component at the $n$th level which is a Seifert fiber space over a disk with exact two exceptional fibers, and $n \geq 1$.
\end{prop}

\begin{proof}

By Lemma \ref{v_0}, $v_0(\alpha)$ is a solid torus, which means $\alpha$ filling on $v_0$ induces a Dehn filling on $v_1$ (i.e.the only one vertex on the 1-level since $v_0$ has only two boundary components). This is the same situation when we discussed about $v_0$, so we can use Lemma \ref{v_0} inductively to verify that $v_{i}$ ($i=1,\dots,n-1$) is a Seifert fiber space over an annulus with exact 1 exceptional fiber, and then the splitting graph is a linear tree shown as Figure \ref{lineartree}. Besides, $v_{i} \cup \dots \cup v_0 \cup \widetilde{X_1}$ is a solid torus, for $i=1,\dots,n-1$. Now we only need to work out the end piece $v_n$.

\begin{figure}[H]
\centering
\includegraphics[width=0.15\textwidth]{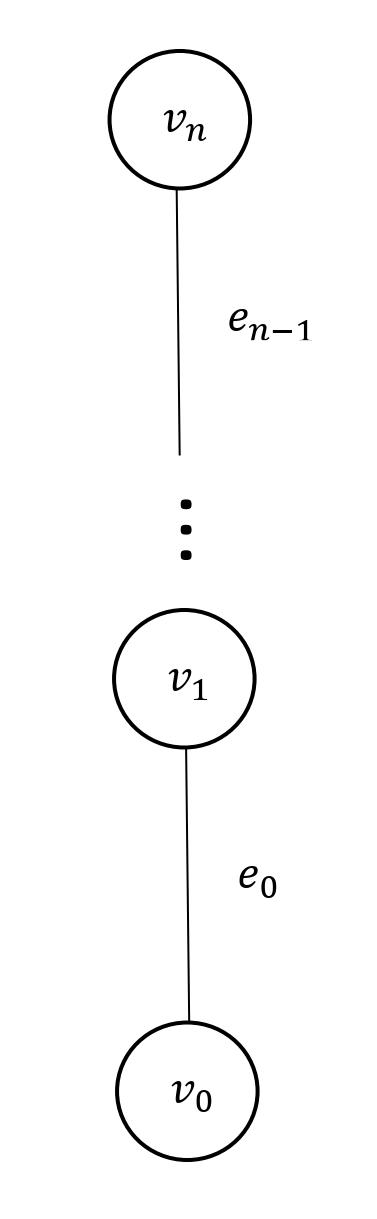}
\caption{\label{lineartree}The splitting graph of $\widetilde{U}$: $v_i$ is a Seifert fiber space over an annulus with exact one exceptional fiber, for $i=0,\dots,n-1$. $v_n$ is a Seifert fiber space over a disk with exact two exceptional fibers.}
\end{figure}

$v_n$ is a small Seifert fiber space with one boundary component. Only small Seifert fiber spaces $M(0,1;(\alpha_1,\beta_1),(\alpha_2,\beta_2))$ and $M(-1,1;)$ have one boundary component. Actually, $M(-1,1;)=M(0,1; (2,1), (2,-1))$. Therefore we can assume $v_n$ is a Seifert fiber space over a disk. It's impossible for $v_n$ to have only one exceptional fiber, for otherwise $v_n$ is a solid torus, and $e_{n-1}$ is not incompressible. So $v_n$ is a Seifert fiber space over a disk with exact two exceptional fibers. Obviously, $n \geq 1$ since $\widetilde{U}$ is not a Seifert fiber space.
\end{proof}

\begin{defn}
The $(p_0, q_0)$-cable of the $(p_1, q_1)$-cable of \dots $(p_k, q_k)$-torus knot is called an {\it iterated knot}, denoted by $[p_0,q_0; p_1, q_1; \dots; p_k, q_k]$, where $p_i \geq 2$, for $i=0,\dots,k$.
\end{defn}

\begin{prop}
\label{complement of an iterated knot}
Let the lens space $\widetilde{b_1}=T_1 \cup T_2$, where $T_i$ is a solid torus, for $i=1,2$. Then $\widetilde{U}=\widetilde{b_1} \setminus N(K)$, where $K$ is an iterated knot $[p_0,q_0; p_1, q_1; \dots ; p_n,q_n]$ in $T_1$ where $p_i \geq 2$ for $i=0,\dots, n$, and $n \geq 1$.
\end{prop}

\begin{proof}
We already have that $v_n$ is a Seifert fiber space over a disk with exact 2 exceptional fibers, and $v_i$ is a Seifert fiber space over an annulus with exact 1 exceptional fiber, for $i=1,\dots,n-1$. $\widetilde{X_1}$ is a solid torus lying in $\widetilde{b_1}$, and it can be regarded as a tubular neighborhood of a knot $K$ in $\widetilde{b_1}$. $\widetilde{U}=\widetilde{b_1}\setminus N(K)$.

Let $r_i$ be the corresponding slope of the Dehn filling on $v_i$ induced by $r_{i-1}$ filling on $v_{i-1}$ where $i=0,\dots,n$ and $r_0=\alpha$. First of all, $v_{n-1}(r_{n-1})$ is a solid torus lying in the lens space $\widetilde{b_1}$ such that $\widetilde{b_1} \setminus v_{n-1}(r_{n-1})= v_n$, which is a Seifert fiber space over a disk with exact 2 exceptional fibers. Then $v_{n-1}(r_{n-1})$ is isotopic to a tubular neighborhood of a fiber, denoted by $F$, of some generalized Seifert fibration of the lens space $\widetilde{b_1}$ over $S^2$. The fiber $F$ is an ordinary fiber, for otherwise $v_{n}$ has at most 1 exceptional fiber which contradicts Proposition \ref{linear tree}. We regard $v_{n-1}(r_{n-1})$ as a tubular neighborhood of a $(p_n,q_n)$-torus knot in one of the solid tori of $\widetilde{b_1}$, without loss of generality $T_1$. We claim that $p_n \geq 2$. If $p_n=0$, then either $\widetilde{U}=\widetilde{b_1} \setminus N(K)$ is reducible, or $\widetilde{b_1}=S^3$, in which case $\widetilde{b_1} \setminus v_{n-1}(r_{n-1})=v_n$ is a solid torus instead of a Seifert fiber space over a disk with 2 exceptional fibers. If $p_n=1$, then $\widetilde{b_1} \setminus v_{n-1}(r_{n-1})=v_n$ is also a solid torus. $v_{n-2}(r_{n-2})$ is a solid torus lying in $v_{n-1}(r_{n-1})$ such that $v_{n-1}(r_{n-1}) \setminus v_{n-2}(r_{n-2})=v_{n-1}$ is a Seifert fiber space over an annulus with exact 1 exceptional fiber. So $v_{n-2}(r_{n-2})$ must lie in $v_{n-1}(r_{n-1})$ as a tubular neighborhood of a $(p_{n-1},q_{n-1})$-torus knot with $p_{n-1} \geq 2$. $p_{n-1} \neq 0$ for the same reason as $p_n \neq 0$, and if $p_{n-1}=1$ then $v_{n-1}$ has no exceptional fiber. Hence, $v_{n-2}(r_{n-2})$ lies in $T_1$ of $\widetilde{b_1}$ as a tubular neighborhood of a $(p_{n-1},q_{n-1})$-cable of $(p_{n},q_{n})$-torus knot with $p_n$, $p_{n-1} \geq 2$.

By induction, $v_0(\alpha)$ lies in $T_1$ of $\widetilde{b_1}$ as a tubular neighborhood of an iterated knot. Also the solid torus $\widetilde{X_1}$ lies in $T_1$ of $\widetilde{b_1}$ as a tubular neighborhood of $K=[p_0,q_0; p_1,q_1; \dots ; \\ p_n,q_n]$ with $p_i \geq 2$, for $i=0, \dots, n$ and $n \geq 1$. $\widetilde{U}=\widetilde{b_1} \setminus N(K)$ is the complement of an iterated knot in $T_1$ of $\widetilde{b_1}$.
\end{proof}

Now we want to find out what kind of iterated knots lying in a lens space admit a surgery to give another lens space. Actually, Gordon \cite{Gordon} has deeply studied on surgeries along an iterated knot in $S^3$ by using Lemma \ref{Gordon1} and \ref{Gordon2}. Here we just use the lemmas of Gordon and results about Seifert fiber space to study Surgeries along the iterated knot $K$ in the lens space $\widetilde{b_1}$.

\begin{prop}
\label{Surgery on iterated knot}
Let $\widetilde{b_1}=L(a,b)=T_1 \cup T_2$, where $(a,b)$ is a pair of relatively prime integers satisfying $0<\frac{b}{a}\leq 1 (or \,\, \frac{b}{a}=\frac{1}{0}=\infty)$, and $T_i$ is a solid torus, for $i=1,2$. Then $\widetilde{U}=L(a,b)\setminus N(K)$, where $K=[p_0, q_0; p_1, q_1]$ in $T_1$ of $\widetilde{b_1}$ with $q_0=2p_1q_1 \pm 1, p_0=2, p_1>1, |aq_1-bp_1|>1$. Only $r=4p_1q_1 \pm 1$ surgery along $K$ can produce another lens space and the manifold obtained by this Dehn surgery is $L(a \pm 4ap_1q_1 \mp 4b p^2_1, b \pm 4aq^2_1 \mp 4b p_1q_1)$.
\end{prop}

\begin{proof}
According to Proposition \ref{complement of an iterated knot}, when $n=1$, the knot $K=[p_0, q_0; p_1, q_1]$ which is a $(p_0, q_0)$-cable of $(p_1, q_1)$-torus knot $T_{p_1,q_1}$ lying in $T_1$ of $L(a,b)$ with $p_1 \geq 2$, $p_0 \geq 2$. By Lemma \ref{Gordon2}, there are three cases.

(1)The surgery slope $r \neq p_0q_0$ and $\neq m/n, \, m=np_0q_0 \pm 1$.

$K(r)$ contains an essential torus, which can not be a lens space.

(2)The surgery slope $r=p_0q_0$.

We get a non-prime manifold.

(3)The surgery slope $r=m/n, \, m=np_0q_0 \pm 1$.

By Lemma \ref{Gordon1} and \ref{Gordon2},
\[K(r)=[L(a,b) \setminus N(T_{p_1,q_1})](m/(np^2_0)).\]
By Lemma \ref{Lens space SFS3},
\[[L(a,b) \setminus N(T_{p_1,q_1})](m/np_0^2) = M(0,0;(p_1, -e), (aq_1-bp_1,ad-be), (m-np^2_0p_1q_1,np^2_0)),\]
where $p_1d-q_1e=1$. We want it to be a lens space, and that happens if and only if one of $|m-np^2_0p_1q_1|$, $|aq_1-bp_1|$ and $|p_1|$ equals 1. In fact, $|aq_1-bp_1|>1$ for the same reason as $p_n=p_1>1$ since a $(p_1,q_1)$-torus knot in $T_1$ of $L(a,b)$ is also a torus knot with winding number $|aq_1-bp_1|$ in $T_2$ of $L(a,b)$. So we have $|m-np^2_0p_1q_1|=1$, and there are two cases.

(\romannumeral1) $m-np^2_0p_1q_1=1$.\\
According to the assumption, we also have $m=np_0q_0 \pm 1$. If $m=np_0q_0+1$, then $np_0(q_0-p_0p_1q_1)=0$. If $n=0$, then $m=1$ and $K(r)=L(a,b)$, while we want it to be another lens space. Besides, $p_0 \geq 2$, and $q_0-p_0p_1q_1 \neq 0$ since $gcd(p_0,q_0)=1$. Thus it is impossible. If $m=np_0q_0-1$, it implies that
\[np_0(q_0-p_0p_1q_1)=2.\]
Therefore $p_0=2$ (since $p_0 \geq 2$), $n=\pm 1$, $q_0-p_0p_1q_1=\pm 1$ so $q_0=2p_1q_1 \pm 1$. $m=\pm p_0q_0-1=\pm4p_1q_1+1$, then $r=\frac{m}{n}=4p_1q_1 \pm 1$. Then Using the formula in Lemma \ref{Lens space SFS1},
\begin{align*}
&[L(a,b) \setminus N(T_{p_1,q_1})](m/np_0^2)\\
&=M(0,0; (p_1,-e), (aq_1-bp_1,ad-be), (1,\pm p^2_0))\\
&=M(0,0;(p_1,-e), (aq_1-bp_1,ad-be \pm p^2_0(aq_1-bp_1)))\\
&=L(a \pm 4ap_1q_1 \mp 4b p^2_1, b \pm 4aq^2_1 \mp 4b p_1q_1).
\end{align*}

(\romannumeral2) $m-np^2_0p_1q_1=-1$.\\
As the argument above, if $m=np_0q_0-1$, then $np_0(q_0-p_0p_1q_1)=0$ which is impossible. Thus $m=np_0q_0+1$, it implies that \[np_0(p_0p_1q_1-q_0)=2.\]
Therefore $p_0=2$ (since $p_0 \geq 2$), $n=\pm 1$, $p_0p_1q_1-q_0=\pm 1$ so $q_0=2p_1q_1 \mp 1$. $m=\pm p_0q_0+1=\pm4p_1q_1-1$, then $r=\frac{m}{n}=4p_1q_1 \mp 1$. Then Using the formula in Lemma \ref{Lens space SFS1},
\begin{align*}
&M(0,0;(p_1,-e), (aq_1-bp_1,ad-be), (1,\mp p^2_0))\\
&=M(0,0;(p_1, -e), (aq_1-bp_1, ad-be \mp p^2_0(aq_1-bp_1)))\\
&=L(a \mp 4ap_1q_1 \pm 4b p^2_1, b \mp 4aq^2_1 \pm 4b p_1q_1).
\end{align*}
When $n=2$, the knot $K=[p_0, q_0;p_1, q_1; p_2, q_2]$ lying in $T_1$ of $L(a,b)$ with $p_i \geq 2$, for $i=0,1,2$. Let $K_1=[p_1, q_1; p_2, q_2]$. Also by Lemma \ref{Gordon2}, there are three cases.

(1)The surgery slope $r \neq p_0q_0$ and $\neq m/n, \, m=np_0q_0 \pm 1$.

$K(r)$ contains an essential torus, which can not be a lens space.

(2)The surgery slope $r=p_0q_0$.

We obtain a non-prime manifold.

(3)The surgery slope $r=m/n, \, m=np_0q_0 \pm 1$.

By Lemma \ref{Gordon1} and \ref{Gordon2},
\[K(r)=[L(a,b) \setminus N(K_1)](m/(np^2_0)).\]
We want it to be a lens space, and that may happen when $m=np^2_0p_1q_1 \pm 1$ by the same argument when $n=1$. Then, by Lemma \ref{Gordon1} and \ref{Gordon2},
\[K(r)=[L(a,b) \setminus N(K_1)](m/(np^2_0))=[L(a,b) \setminus N(T_{p_2,q_2})](m/(np^2_0p^2_1)).\]
By Lemma \ref{Lens space SFS3},
\begin{gather*}
[L(a,b) \setminus N(T_{p_2,q_2})](m/(np^2_0p^2_1))\\
= M(0,0;(p_2,-e), (aq_2-bp_2, ad-be), (m-np^2_0p^2_1p_2q_2, np^2_0p^2_1)),
\end{gather*}
where $p_2d-q_2e=1$. It can be a lens space if and only if $|m-np^2_0p^2_1p_2q_2|=1$, since if one of $|aq_2-bp_2|$ and $|p_2|$ equals 1, then $v_n=v_2$ has only one exceptional fiber. So far we have 3 equations shown as the following:
\begin{align}
&m=np_0q_0 \pm 1             \label{eq9}\\
&m=np^2_0p_1q_1 \pm 1        \label{eq10}\\
&m=np^2_0p^2_1p_2q_2 \pm 1   \label{eq11}
\end{align}
By Equation (\ref{eq10}) and (\ref{eq11}), we have $np^2_0p_1(q_1-p_1p_2q_2)=0$ or $\pm2$. It is impossible since $p_0, p_1 \geq 2$, $gcd(p_1,q_1)=1$ and $n \neq 0$ for otherwise $K(r)=L(a,b)$. Therefore, when $n=2$, surgeries along $K$ can not produce another lens space.

Using the same argument as $n=2$, we have $n=3,4,\dots$ is also impossible. Then it concludes the proposition.
\end{proof}

It not hard to find that $\widetilde{U}$ satisfying the equations \ref{eq7} and \ref{eq8} is contained in $\widetilde{O}$ satisfying the equations \ref{eq3} and \ref{eq4} when $\widetilde{O}$ is irreducible toroidal but not Seifert fibered. Actually $\widetilde{U}$ is a special case of the previous $\widetilde{O}$ (i.e.$p=2$). So we can easily get the following theorem.

\begin{thm1}
Suppose
\begin{align*}
N(U+X_1)&=b_1          \\
N(U+X_2)&=b_2,
\end{align*}
where $X_1$ and $X_2$ are rational tangles, $U$ is an algebraic tangle but not a generalized Montesinos tangle, $b_1$ and $b_2$ are 2-bridge links with $b_1 \neq b_2$. The system of tangle equations has solutions if and only if one of the following holds:

(1) There exist 2 pairs of relatively prime integers $(a,b)$, $(p_1,q_1)$ satisfying $0<\frac{b}{a}\leq 1 (or \,\, \frac{b}{a}=\frac{1}{0}=\infty), p_1>1$, and $|aq_1-bp_1|>1$ such that $b_1=b(a,b)$ and $b_2=b(a+4ap_1q_1-4b p^2_1, b+4aq^2_1-4b p_1q_1)$. Solutions up to equivalence are shown as the following:
\begin{figure}[H]
\centering
\subfigure[]{
\includegraphics[width=0.24\textwidth]{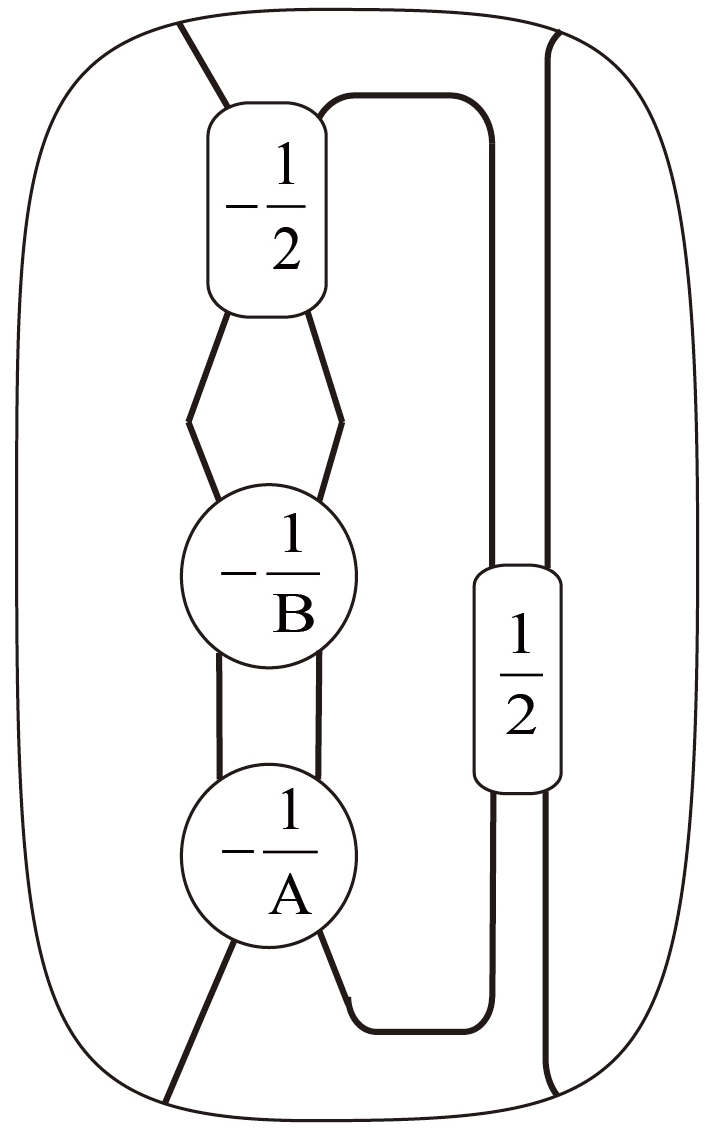}}\quad
\subfigure[]{
\includegraphics[width=0.24\textwidth]{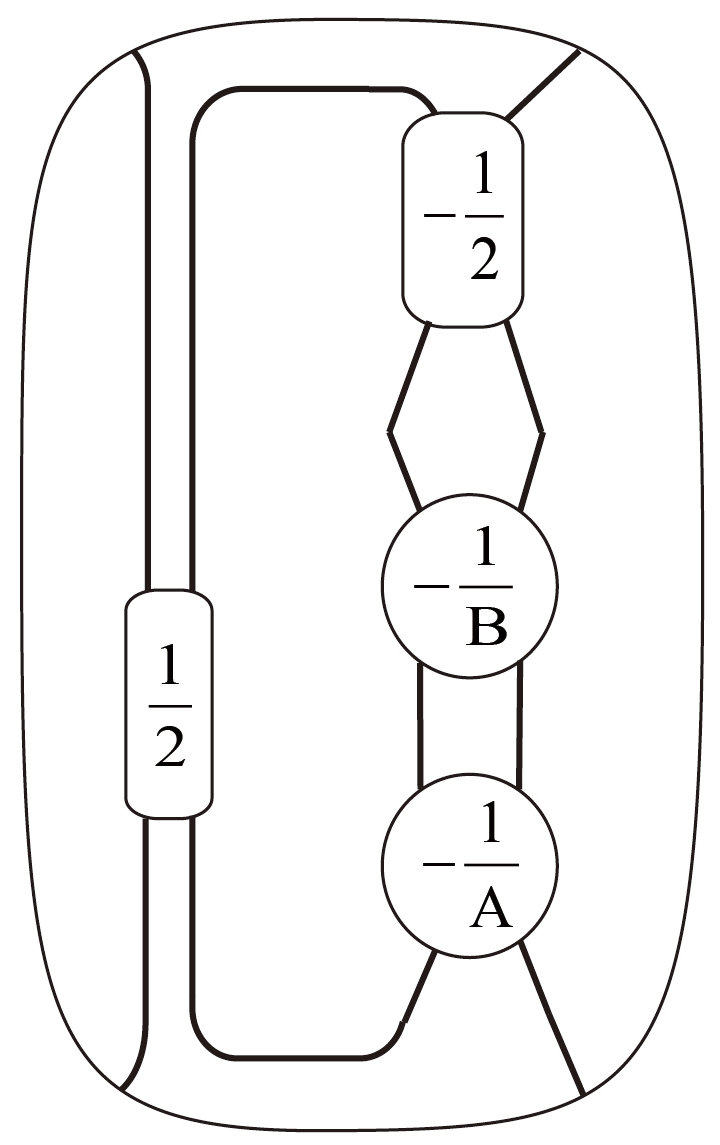}}\quad
\subfigure{
\includegraphics[width=0.16\textwidth]{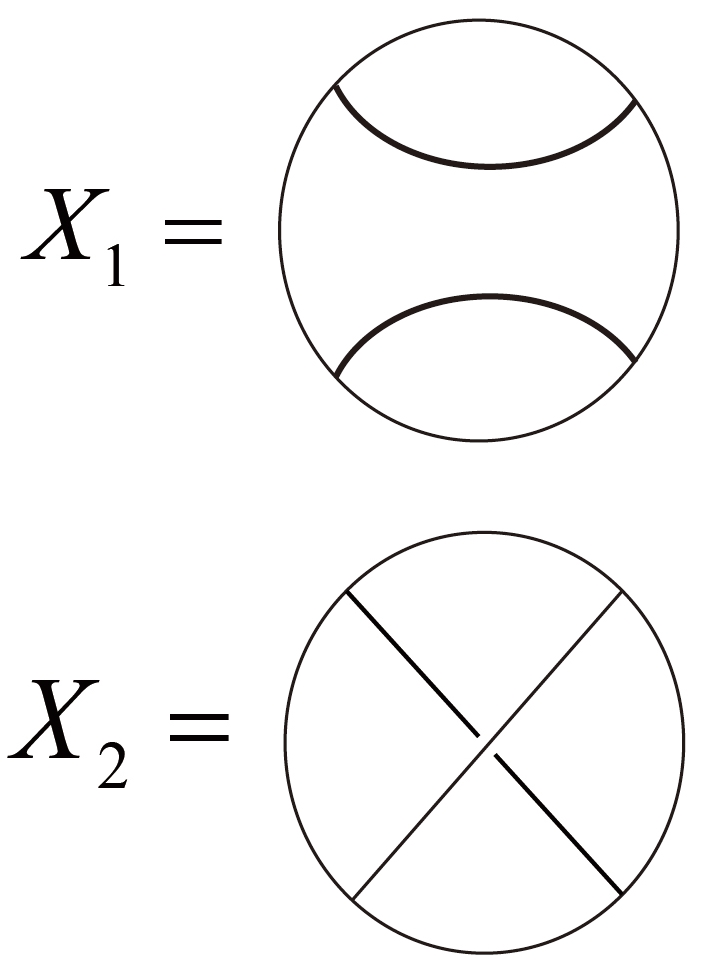}}
\caption{$U=$ the tangle in (a) or (b) where $A=\frac{-e}{p_1}$, $B=\frac{ad-be}{aq_1-bp_1}$ (or $A=\frac{ad-be}{aq_1-bp_1}$, $B=\frac{-e}{p_1}$), and $p_1d-q_1e=1$ with $d,e\in \mathbb{Z}$. $X_1=0$-tangle and $X_2=-1$-tangle. (Note that choosing different $d$ and $e$ such that $p_1d-q_1e=1$ has no effect on the tangle $U$.)}
\end{figure}
(2) There exist 2 pairs of relatively prime integers $(a,b)$, $(p_1,q_1)$ satisfying $0<\frac{b}{a}\leq 1 (or \,\, \frac{b}{a}=\frac{1}{0}=\infty), p_1>1$ and $|aq_1-bp_1|>1$ such that $b_1=b(a,b)$ and $b_2=b(a-4ap_1q_1+4b p^2_1, b-4aq^2_1+4b p_1q_1)$. Solutions up to equivalence are shown as the following:
\begin{figure}[H]
\centering
\subfigure[]{
\includegraphics[width=0.24\textwidth]{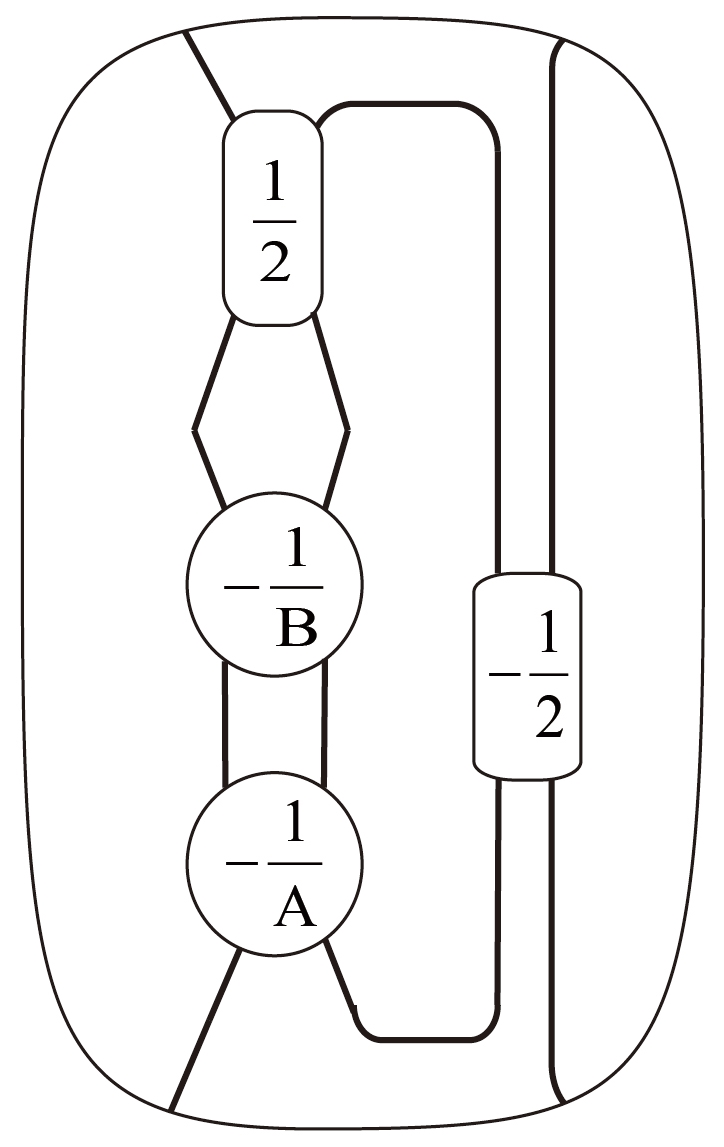}}\quad
\subfigure[]{
\includegraphics[width=0.24\textwidth]{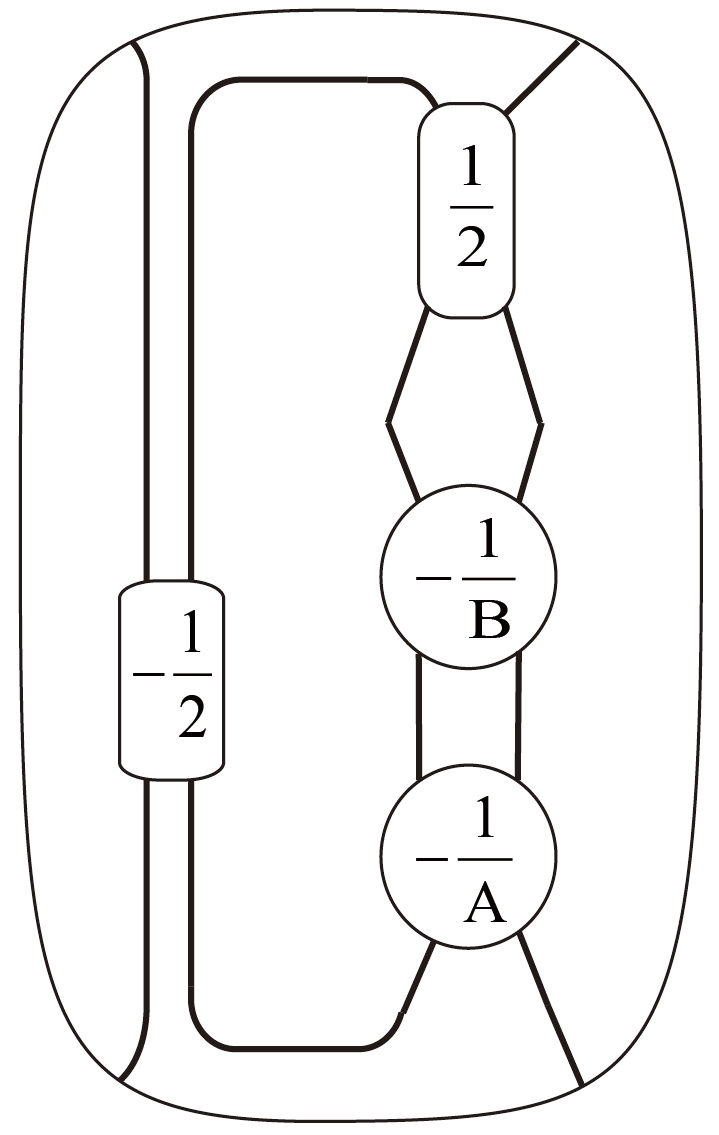}}\quad
\subfigure{
\includegraphics[width=0.16\textwidth]{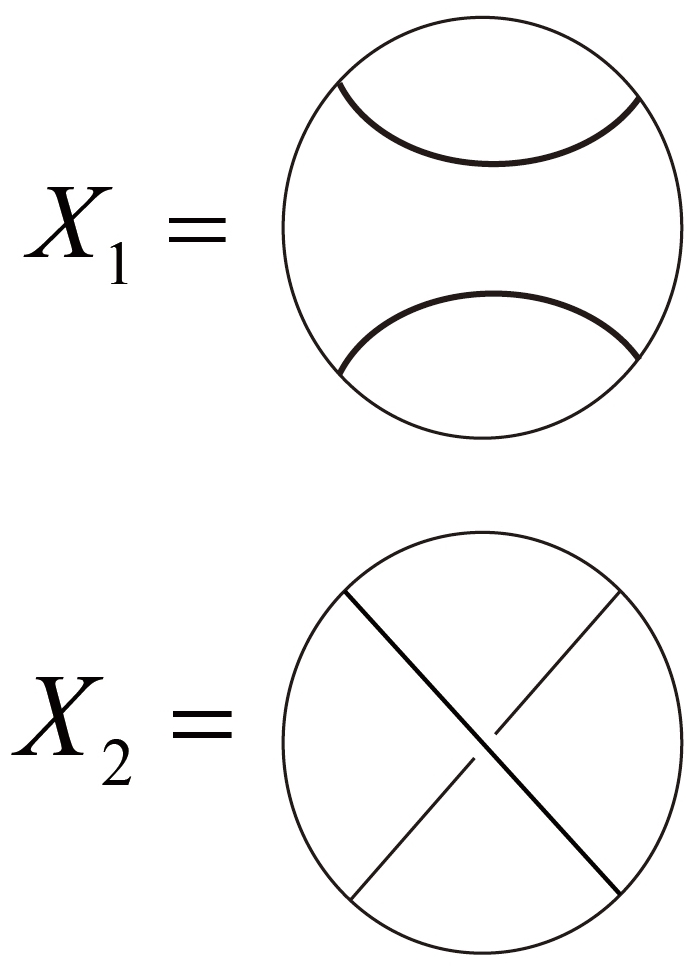}}
\caption{$U=$ the tangle in (a) or (b) where $A=\frac{-e}{p_1}$, $B=\frac{ad-be}{aq_1-bp_1}$ (or $A=\frac{ad-be}{aq_1-bp_1}$, $B=\frac{-e}{p_1}$), and $p_1d-q_1e=1$ with $d,e\in \mathbb{Z}$. $X_1=0$-tangle and $X_2=1$-tangle. (Note that choosing different $d$ and $e$ such that $p_1d-q_1e=1$ has no effect on the tangle $U$.)}
\end{figure}
\end{thm1}

\begin{proof}
We just use the same method as in Theorem \ref{mainthm}. Suppose $b_1=b(a,b)$ for a pair of relatively prime integers $(a,b)$ satisfying $0<\frac{b}{a}\leq 1 (or \,\, \frac{b}{a}=\frac{1}{0}=\infty)$. According to Proposition \ref{Surgery on iterated knot}, $\widetilde{U}$ is the complement of the iterated knot $K=[p_0, q_0; p_1, q_1]$ with $q_0=2p_1q_1 \pm 1, p_0=2$, $p_1>1, |aq_1-bp_1|>1$ lying in one of the solid torus of the lens space $L(a,b)$. This is a special case of $\widetilde{O}$ in Theorem \ref{mainthm}, i.e.$p=2$ for $\widetilde{O}$. Therefore, we can find all the tangles whose double branched cover are $\widetilde{U}$, by letting $p=2$ in all the solutions for $O$ in Theorem \ref{mainthm}.

The only difference is the surgery slopes on $K$, since here we want to obtain a lens space instead of a connected sum of two lens spaces. Obviously, only $\infty$-surgery along $K$ gives the original $L(a,b)$. By Proposition \ref{Surgery on iterated knot}, only $r=4p_1q_1 \pm 1$-surgery produces another lens space. By carefully studying the images of these slopes under the covering maps, we obtain the pairs of rational tangle solutions up to equivalence as above.
\end{proof}

\end{document}